\documentclass[11pt,reqno]{amsart}
\usepackage{graphicx} %
\usepackage[margin=1in]{geometry}

\usepackage{euflag}
\usepackage{amsmath, amssymb}
\usepackage{mathabx}

\usepackage[shortlabels]{enumitem}

\DeclareMathAlphabet{\lcal}{U}{dutchcal}{m}{n}

\allowdisplaybreaks

\usepackage{xcolor} %
\definecolor{darkred}{rgb}{0.5,0,0} %
\definecolor{darkblue}{rgb}{0,0,0.5} %

\usepackage[pagebackref,colorlinks,linkcolor=darkred,citecolor=darkblue,urlcolor=darkblue,hypertexnames=true]{hyperref}

\usepackage{lmodern}
\usepackage[T1]{fontenc}

\makeatletter
\renewcommand{\subsubsection}{\@startsection{subsubsection}{3}{\z@}%
	{.5\linespacing\@plus.7\linespacing}{-.5em}%
	{\normalfont\fontseries{bx}\selectfont}} %
\makeatother

\usepackage{mathrsfs}
\usepackage[normalem]{ulem}

\newcommand{\spann}{\operatorname{span}}

\newcommand{\cA}{\mathcal A}
\newcommand{\cB}{\mathcal{B}}
\newcommand{\cD}{\mathcal D}
\newcommand{\cE}{\mathcal E}
\newcommand{\cF}{\mathcal{F}}

\newcommand{\cH}{\mathcal H}

\newcommand{\cK}{\mathcal{K}}

\newcommand{\cM}{\mathcal{M}}
\newcommand{\cN}{\mathcal{N}}

\newcommand{\cS}{\mathcal{S}}
\newcommand{\cT}{\mathcal{T}}
\newcommand{\cU}{\mathcal{U}}

\newcommand{\la}{\langle}
\newcommand{\ra}{\rangle}
\newcommand{\II}[1]{\left\langle x \right\rangle_{#1}}

\newcommand{\vg}{{\tt{g}}}

\newcommand{\C}{\mathbb C}
\newcommand{\N}{\mathbb N}
\newcommand{\bA}{\mathbb{A}}

\newcommand{\real}{\operatorname{real}}

\newcommand{\df}[1]{\textit{#1}}

\newcommand{\ad}{\lcal{d}}
\newcommand{\Ftgd}{\mathcal{F}^2_{\vg,\ad}}  %
\newcommand{\ME}{\mathbb{E}}  %
\newcommand{\SG}{G}  %

\newcommand{\wtC}{\widetilde{\Sigma}}  %
\newcommand{\oS}{\mathscr{S}} %
\newcommand{\oC}{\mathfrak{A}} %
\newcommand{\UCP}{\operatorname{UCP}} %
\newcommand{\Ddom}{\mathscr{D}}
\newcommand{\wtpi}{\widetilde{\pi}}
\newcommand{\penM}{\Lambda}
\newcommand{\vd}{{\tt{d}}}
\newcommand{\vh}{{\tt{h}}}
\newcommand{\VT}{T}

\newcommand{\range}{\operatorname{range}}

\newcommand{\nmL}{\mathscr{L}}
\newcommand{\snmL}{\mathfrak{L}}

\newcommand{\trace}{\operatorname{tr}}
\newcommand{\groupW}{\mathbb{W}}

\newcommand{\BO}{B}
\newcommand{\HE}{\cE}
\newcommand{\HG}{\mathcal{G}}

\newcommand{\Zn}{\Z_n}
\newcommand{\CZ}{\C[\Z_n]}
\newcommand{\POVM}{\operatorname{POVM}}
\newcommand{\POVMn}{\POVM(n)}

\newcommand{\R}{\mathbb{R}}
\newcommand{\cbast}{\check{\Asterisk}}

\newcommand{\cx}{\lcal{x}}
\newcommand{\Z}{\mathbb{Z}}
\newcommand{\rep}{\tau}

\newcommand{\povm}{\operatorname{povm}}

\newcommand{\ch}{\lcal{h}}
\newcommand{\cp}{\lcal{p}}
\newcommand{\cq}{\lcal{q}}

\newcommand{\zz}{z}
\newcommand{\pp}{\widecheck{p}}   
\newcommand{\fq}{\mathfrak{q}}
\newcommand{\fqy}{y}

\newcommand{\Cly}{\C\langle \fq \rangle}

\newcommand{\affL}{\widehat{L}}
\newcommand{\affp}{\hat{p}}
\newcommand{\affq}{\hat{q}}
\newcommand{\affr}{\hat{r}}

\newcommand{\affbA}{\widehat{\bA}}

\newcommand{\notV}{\mathbf V}

\newcommand{\mell}{\mu}
\newcommand{\tcH}{\cK}

\makeatletter
\newtheorem*{rep@thm}{\rep@title}
\newcommand{\newreptheorem}[2]{%
	\newenvironment{rep#1}[1]{%
		\def\rep@title{#2 \ref{##1}}%
		\begin{rep@thm}}%
		{\end{rep@thm}}}
\makeatother

\newtheorem{thm}{Theorem}[section]

\newtheorem{corollary}[thm]{Corollary}

\newtheorem{lemma}[thm]{Lemma}

\newtheorem{proposition}[thm]{Proposition}
\newtheorem{prop}[thm]{Proposition}

\newtheorem{thmA}{Theorem}

\newtheorem{remark}[thm]{Remark}

\theoremstyle{plain}
\newtheorem*{theorem*}{Theorem}

\numberwithin{equation}{section}

\begin{document}
	
	\title[Positivstellensatz for Matrix Convex Sets]{Operator-Valued Positivstellensätze on Matrix Convex Sets and Free Products of Finite Abelian Groups}

	\author[A.\ Jindal]{Abhay Jindal${}^{1,Q}$}
	\address{Abhay Jindal, Faculty of Mathematics and Physics, University of Ljubljana, Slovenia}
	\email{abhay.jindal@fmf.uni-lj.si}
	\thanks{${}^1$Supported by the Slovenian Research Agency 
		program P1-0222 and grant J1-50002. AJ thanks
		l’\'Ecole Polytechnique and Inria for
		hospitality during the preparation of this manuscript. }
	
	\author[I.\ Klep]{Igor Klep${}^{2,Q}$}
	\address{Igor Klep, Faculty of Mathematics and Physics, University of Ljubljana 
		\& Famnit, University of Primorska, Koper 
		\& Institute of Mathematics, Physics and Mechanics,
		Ljubljana, Slovenia}
	\email{igor.klep@fmf.uni-lj.si}
	\thanks{${}^2$Supported by the Slovenian Research Agency 
		program P1-0222 and grants J1-50002,
		N1-0217,
		J1-60011, J1-50001, J1-3004 and J1-60025. Partially supported by the Fondation de l’\'Ecole polytechnique as part
		of the Gaspard Monge Visiting Professor Program. IK thanks
		\'Ecole Polytechnique and Inria Paris Saclay for
		hospitality during the preparation of this manuscript.}
	
	\author[S.\ McCullough]{Scott McCullough}
	\address{Scott McCullough, Department of Mathematics\\
		University of Florida\\ Gainesville} %
	
	\email{sam@math.ufl.edu}
	
	\thanks{}

	\thanks{${}^Q$This work was performed within the project COMPUTE, funded within the QuantERA II 
		Programme that has received funding from the EU's H2020 research and innovation programme under the GA No 101017733 {\normalsize\euflag}}

	\begin{abstract}
		We prove a Positivstellensatz for operator-valued noncommutative polynomials that are positive on matrix convex sets.
		Specifically, let \(p\in  B(\cH)\otimes \C\langle x\rangle\) be an operator-valued polynomial of degree at most \(2\vd+1\), where \(\cH\) is separable and infinite-dimensional. Let
		$	L(x)=I+\sum_{j=1}^{\vg} A_j x_j $ be a monic linear operator pencil, and let 
		$	\cD_L=\{X\mid L(X)\succeq 0\} $
		be the associated matrix convex set. We show that $p$ is positive on $\cD_L$ if and only if
		\[
		p \ =\ r^*r+q^*\pi(L)q,
		\]
		where $ q,r\in  B(\cH)\otimes \C\langle x\rangle$ have degree at most $\vd$, and $\pi$ is a unital completely positive map on the operator system generated by the coefficients of $L$. 
		The proof combines a Hahn–Banach separation argument with a tailored GNS construction. The main challenge in implementing the  GNS  construction in the present context is that the separation occurs in the product ultraweak topology, so boundedness of the resulting GNS operators is not automatic. We first handle the case of bounded matrix convex sets, using the closedness of the cone of weighted squares (in the product ultraweak topology) as the key technical input, and then pass to the general unbounded case via an approximation argument.
		
		\smallskip
		
		Finally, we apply this convex Positivstellensatz to prove an operator-valued noncommutative Fej\'er--Riesz theorem on free products of finite abelian groups. The key additional ingredients are the universal $\ast$-algebra $\mathrm{povm}(n)$ associated with POVMs, a `perfect' Positivstellensatz for $\mathrm{povm}(n)$, and Boca's theorem on free products of completely positive maps. As a consequence, every positive operator-valued trigonometric polynomial on a free product of finite abelian groups admits a sum-of-squares factorization with explicit complexity bounds.
	\end{abstract}
	
	\maketitle
	
	\newpage
	
	\setcounter{tocdepth}{3}
	\begin{list}{}{%
			\setlength{\leftmargin}{1.5cm}%
			\setlength{\rightmargin}{1.5cm}%
			\setlength{\listparindent}{0pt}
			\setlength{\itemsep}{0pt}
			\setlength{\parsep}{2.5pt}
		}
		\item
		\tableofcontents
	\end{list}
	
	\newpage
	
	\section{Introduction}

	Positivity and sums of squares lie at the heart of operator theory and real algebraic geometry. In the commutative setting, the quest for positivity certificates via sums of squares goes back to Hilbert’s 17th problem in 1900 and is closely intertwined with the development of moment theory and the Positivstellensätze of Schmüdgen \cite{Scm91}
	and Putinar \cite{Put93}; for classical results and modern treatments see \cite{BCR98, Mar08, Sc24}.
	
	\smallskip
	
	In the 21st century, ideas from linear systems theory and optimization \cite{SIG98,dOHMP09, WM21}, quantum physics \cite{brunner,NPA07,NPA08}, and free probability \cite{MS17, VDN92} helped drive the development of the free, or noncommutative, analog into a broad area of noncommutative function theory \cite{KVV14,MS11,AM15,BMV16,PTD22,Voi04,Voi10}. Among its central themes are factorization and Positivstellens\"atze for noncommutative polynomials. An early milestone is Helton’s theorem showing that positive scalar-valued noncommutative polynomials are sums of squares \cite{Hel02}
	(see also \cite{McC01}). 
	Recently, Vol\v ci\v c \cite{Vol21} established a proper noncommutative analog of Artin's solution to Hilbert's 17th problem.
	See also \cite{HM04,HMP04,Po95,JM12,JMS21} and the references therein for further developments.
	
	\smallskip
	
	In the noncommutative setting, sums of squares and positivity are often much more closely aligned than in the classical one.
	This phenomenon is illustrated by convex Positivstellens\"atze \cite{HKM12,HKM17}, which give algebraic certificates for polynomials that are positive on free spectrahedra, that is, on sets defined by linear matrix inequalities \cite{HM12,Zal17}. Such results now play a central role in free analysis and matrix convexity \cite{Kri19,ANT19,Pas22,Vol24}, and are closely tied to the theory of completely positive maps and operator systems \cite{Pau03,EW97,DDSS17,FHL18,EPS24}. Similar rigidity is observed in various noncommutative factorization theorems in the style of Fejér--Riesz \cite{DR10,GW05}, e.g., for positivity in free group algebras \cite{McC01,BT07,Oz13} and virtually-free groups \cite{NT13,KLM}. 
	We also note limitations of algorithmic approaches to noncommutative positivity: positivity is undecidable in certain tensor-product settings \cite{MSZ23,Lin+}, and related non-attainment phenomena occur in the commuting-operator setting \cite{FKM25}.
	
	\subsection{Main results} 
	Motivated by these developments, this paper considers positivity for operator-valued noncommutative polynomials from two complementary perspectives. 
	Our first main result, Theorem \ref{t:main1}, is a convex Positivstellensatz. Roughly speaking, it says that if an operator-valued noncommutative polynomial is positive on a matrix convex set (defined by a linear operator pencil), then it admits a weighted sum-of-squares certificate of optimal half-degree type. The certificate involves a ucp map applied to the pencil, reflecting the operator-valued nature of the problem.
	The second main result, Theorem \ref{t:main2}, applies Theorem~\ref{t:main1} to free products of finite abelian groups and yields an operator-valued noncommutative Fejér--Riesz theorem: every positive trigonometric polynomial on such a free product admits a representation as a sum of hermitian squares, together with explicit bounds on the complexity of that representation. 
	Connections between these results and related work in the literature are outlined in remarks accompanying their statements and in Subsection~\ref{ss:new}.
	
	\smallskip
	
	We are grateful to Mehta--Slofstra--Zhao \cite{MSZ} for communicating to us an argument that plays a key role in the proof of Theorem~\ref{t:main2}.
	
	\subsubsection{Noncommutative polynomials and linear pencils}
	Fix  a positive integer \df{$\vg$}.  Let \df{$\la x \ra$}  denote the free monoid on the $\vg$ letters of the alphabet $x = \{ x_{1}, \dots, x_{\vg}\}.$ Its multiplicative identity is the empty word $\varnothing.$ We endow $\la x\ra$ with the \df{graded lexicographic order}. The length of a word $w \in \la x \ra$ is denoted by \df{$|w|$}. The set of all elements (words) of $\la x \ra$ of length (or degree) at most $\ad$ is \df{$\la x \ra_{\ad}$}. Its cardinality %
	is \index{$N(\ad)$} $ N(\ad) \ =\  \sum_{i=0}^\ad \vg^i .$
	
	\smallskip
	
	Unless explicitly stated otherwise, \df{$\cH$} will be a fixed complex separable (infinite-dimensional) Hilbert space. Let  \df{$ B(\cH)$} denote the space of all bounded linear operators on $\cH$, and 
	let \df{$\cA$} denote the free semigroup $ B(\cH)$-algebra on $x;$ that is, $\cA=
	B(\cH)\otimes \C\la x\ra =
	B(\cH)\la x\ra .$ An element $p$ of $\cA$ is of the form, 
	\begin{equation}\label{eq:poly}
		p \ =\  \sum_{w\in\la x\ra}^{\rm finite} P_{w} w,
	\end{equation}
	where
	$P_{w} \in  B(\cH),$ and is referred to as an (operator-valued) polynomial in $x.$ Let \df{$\C\la x\ra_{\ad}$} $\subset \C\la x\ra$ denote the complex-valued polynomials of degree at most $\ad$ and \df{$\cA_{\ad}$} denote the elements of $\cA$ of degree at most $\ad$.
	
	\smallskip
	
	Equip $\cA$ with the involution \df{$^*$}: on letters, $x_{j}^{*} = x_{j},$ on a word $w = x_{i_1} \cdots x_{i_n} \in \la x \ra,$ 
	\[
	w^{*} \ =\  x_{i_n} \cdots x_{i_1};
	\]
	and, on a polynomial $p$ as in \eqref{eq:poly}, \index{$p^*$}
	\[
	p^{*} \ =\  \sum P_{w}^{*} w^{*},
	\]
	where $P_{w}^{*}$ is the adjoint of the operator $P_{w}$ in $ B(\cH).$
	
	\smallskip
	
	Let $X = (X_{1}, \dots, X_{\vg})$ be a tuple of bounded operators on some Hilbert space.  The 
	\df{evaluation} of $p$ at $X$ is defined as 
	\[
	p(X) \ =\  \sum P_{w} \otimes X^{w},
	\]
	where $X^{w} = X_{i_1} \dots X_{i_n}$ for $w = x_{i_1} \dots x_{i_n}.$ In general, $p(X)^{*}$ (the adjoint of $p(X)$) and $p^{*}(X)$ are not the same. They coincide if $X$ is a tuple of self-adjoint operators. 
	
	\smallskip
	
	As a special case of an operator-valued polynomial,   
	let $\tcH$ be a Hilbert space, and let $L$ 
	denote the linear operator pencil (affine linear polynomial) %
	\begin{equation*} 
		L(x) \ =\ P_{0} + \sum\limits_{j=1}^{\vg} P_{j} x_{j},
	\end{equation*}
	where $P_{0}, \ldots, P_{\vg}$ are bounded self-adjoint operators on $\tcH.$
	In the case $P_0=I_{\tcH},$ the polynomial $L$ is a {monic linear operator pencil}. 
	
	\smallskip
	
	For a bounded operator $T$ on a Hilbert space, the notation $T \succeq 0$ means that the operator $T$ is positive semidefinite (psd). The operator inequality
	\[
	L(X) \ :=\ P_{0} \otimes I + \sum\limits_{j=1}^{\vg} P_{j} \otimes X_{j} \ \succeq \ 0 
	\]
	is called a \df{linear operator inequality} \df{(LOI)}. Let \df{$\cD_{L}$} denote the collection of all $\vg$-tuples of self-adjoint matrices $X = (X_{1}, \dots, X_{\vg})$ of any order such that $L(X) \succeq 0.$ We say that $\cD_{L}$ is bounded if there exists a natural number $N$ such that 
	$\sup\{\|X_{j} \| \,:\,X \in \cD_{L} \} \ \leq \ N $
	for all $j = 1,\dots, \vg,$ where $\| \cdot \|$ denotes the operator norm. 
	
	\smallskip
	
	Throughout this article, unless explicitly stated otherwise, we fix a monic linear pencil \index{$L(x)$}
	\begin{equation}
		\label{eq:L=LA}
		L(x) \ =\  I + \sum_{j=1}^\vg \bA_j x_j,
	\end{equation}
	where the \df{$\bA_j$} are self-adjoint operators on the Hilbert space \df{$\tcH$}.
	Let $\oS_L \subset {B}(\tcH)$ \index{$\oS_L$} denote the (unital) operator system spanned by 
	$\bA_{1}, \dots, \bA_{\vg}$, and let \df{${C}^{*}(\oS_{L})$} $\subset {B}(\tcH) $ denote the $C^{*}$-algebra generated by $\oS_{L}$. Note that $\oS_{L}$ is finite-dimensional, and {${C}^{*}(\oS_{L})$ is separable. We write \df{$\UCP(\oS_{L}, {B}(\cH ))$} for the set of unital completely positive (ucp) maps from $\oS_{L}$ into ${B}(\cH )$. 
		For  $\pi \in \UCP (\oS_{L},  B(\cH))$, the linear pencil
		\[
		\pi(L) \ :=\  \pi(I) + \sum\limits_{j=1}^{\vg} \pi (\bA_{j}) x_{j}
		\]
		is monic. If $X \in \cD_{L},$ then $X \in \cD_{\pi(L)}.$ Indeed, for such $X,$  
		\[
		\pi(L)(X) \ =\ \pi(I) \otimes I + \sum\limits_{j=1}^{\vg} \pi (\bA_{j}) \otimes  X_{j} \ =\ (\pi \otimes {\rm id})(L(X)) \ \succeq \  0.
		\]

		We are now in a position to present our first main result. 
		By the Effros-Winkler Hahn-Banach theorem \cite{EW97,HM12} combined with a routine density argument, every closed matrix
		convex set is of the form $\cD_L$ for a LOI $L$. Accordingly, Theorem \ref{t:main1}
		yields a Positivstellensatz for operator-valued polynomials that are positive on a closed matrix convex set:
		
		\begin{thmA}\label{t:main1}
			Let $\cH$ be a separable infinite-dimensional Hilbert space.
			Let $p \in  B(\cH) \otimes \C\la x\ra$ be an operator-valued polynomial of degree at most $2\vd+1$, and let 
			\[
			L \ =\  I + \sum_{j=1}^{\vg} \bA_{j} x_{j} \ \in\   B(\tcH) \otimes \C\la x\ra
			\]
			be a monic linear pencil.     
			Then the following are equivalent:
			\begin{enumerate}[\rm (i)]\itemsep=5pt
				\item For any $n \in \mathbb{N}$ and any $\vg$-tuple of self-adjoint matrices $X = 
				(X_{1}, \ldots, X_{\vg}) \in M_{n}(\C)^{\vg},$ $p(X) \succeq 0$ whenever $L(X) \succeq 0;$  
				\item There exist $q,r \in \cA_{\vd}$ and a ucp map $\pi: \oS_{L} \to  B(\cH)$ such that 
				\begin{equation*}\label{eq:wsos}
					p \ =\ r^{*}r + q^{*} \pi(L) q,
				\end{equation*}
				where $\oS_{L} \subset  B(\tcH)$ is the unital operator system spanned by $\bA_{1}, \ldots, \bA_{\vg}.$
			\end{enumerate}
		\end{thmA}

		\begin{remark} \rm
			Several remarks related to Theorem \ref{t:main1} are in order.
			
			\begin{enumerate}[\rm (a)]\itemsep=5pt
				\item Theorem \ref{t:main1} is stated and proved under the assumption that $\cH$ is infinite  dimensional, so the finite-dimensional cases are not obtained by a direct specialization of Theorem \ref{t:main1}.\looseness=-1
				
				\item Nevertheless, when both $\cH$ and $\tcH$ are finite-dimensional, one recovers \cite[Theorem 1.1]{HKM12} after an additional argument; see Theorem~\ref{c:HKM12} item~\ref{i:cHKM:iii}. Likewise, when $\cH$ is finite-dimensional and $\tcH$ is arbitrary, one recovers \cite[Theorem 1.5]{Zal17}; {see Theorem~\ref{c:HKM12}\ref{i:cHKM:i}.}
				
				\item No assumption is imposed on the dimension of $\tcH$ in Theorem \ref{t:main1}. However, if $\tcH$ is finite-dimensional, then one can obtain the sharper representation
				\[
				p \ =\ r^*r+q^*(I_\cE\otimes L)q,
				\]
				for an auxiliary Hilbert space $\cE$; see Theorem~\ref{c:HKM12}\ref{i:cHKM:ii}.
				
				\item Theorem \ref{t:main1} is proved in Section \ref{sec:proofoft1}. 
				The following reformulation of Theorem \ref{t:main1} is convenient.
				Let $\wtC_{\vd,L}$ denote the \df{cone of weighted squares} of polynomials of degree at most $\vd,$
				\begin{equation*} 
					\wtC_{\vd,L} \ :=\  \{\, r^{*}r + q^{*} \pi(L) q \ : \  r,q \in \cA_{\vd},\  \pi \in \UCP(\oS_{L},  B(\cH))\,\} \ \subset \ \cA_{2\vd+1}.
				\end{equation*}
				With this notation, Theorem~\ref{t:main1} says if $p\in \cA_{2\vd+1},$ then $p\succeq0$ on  $\cD_L$ if and only if $p\in \wtC_{\vd,L}.$
			\end{enumerate}
		\end{remark}
		
		\def\bG{\mathbb G}
		
		\subsubsection{Free products of finite abelian groups} 
		Fix a positive integer \(m\). 
		Let
		\[
		\groupW \ =\  \mathbb{G}_{1}\ast \mathbb{G}_{2} \ast \cdots \ast \mathbb{G}_{m},
		\]
		be the free product of finite abelian groups \(\mathbb{G}_{1},\dots,\mathbb{G}_{m}\).
		
		\smallskip
		
		Every nontrivial \(w \in \groupW\) admits a unique representation as a reduced word, i.e., $w$ is of the form
		\[
		w \;=\; g_1 g_2 \cdots g_k,
		\]
		where, for each \(\ell=1,\dots,k\), one has \( g_\ell \in \mathbb{G}_{i_\ell}\setminus\{e\}, \)
		and consecutive letters come from different factors, that is,
		$
		i_\ell \neq i_{\ell+1}, \; \ell=1,\dots,k-1.
		$
		The \df{extent} of \(w\) is \(k\).

		\smallskip
		
		Let $\cE$ be any separable (finite- or infinite-dimensional) Hilbert space. An element $p$ of $  B(\cE) \otimes \C[\groupW]$ is an \df{operator-valued polynomial} of the form \eqref{eq:poly} with $P_{w} \in  B(\cE)$ and each $w \in \groupW.$
		The \df{extent} 
		of $p$ is the largest extent of a (reduced) word appearing in the sum in equation~\eqref{eq:poly}. 
		
		\smallskip
		
		There is a natural \df{involution} \df{${}^\ast$} on $\groupW.$ On a word $w=g_{1} g_{2}\cdots g_{k} \in \groupW,$
		\[
		w^{*} \ =\ g_{k}^{-1} \cdots g_{2}^{-1} g_{1}^{-1}.
		\]
		This involution extends to
		$  B(\cE) \otimes \C[\groupW]$ by linearity,
		\[
		p^* \ =\  \sum P_w^*\, w^*,
		\] 
		where $P_{w}^{*}$ is the adjoint of the operator $P_{w}$ in $ B(\cE)$ and doing so makes $  B(\cE) \otimes \C[\groupW]$ a $\ast$-algebra.  A polynomial $p$ is \df{hermitian} if $p^*=p.$ In particular, if $p\in  B(\cE) \otimes \C[\groupW],$ then $p+p^*$ is hermitian.
		
		\smallskip
		
		Given a unitary representation \(\rep\) of \(\groupW\) on a Hilbert space, the \df{evaluation} of a polynomial \(p \in B
		(\cE) \otimes \C[\groupW]\) at \(\rep\) is defined by
		\[
		p(\rep)\ :=\ \sum_{w\in \groupW}^{\text{\tiny finite}} P_w \otimes \rep(w).
		\]
		Note that
		\[
		p(\rep)^*
		\ =\ 
		\sum_{w\in \groupW}^{\text{\tiny finite}} P_w^* \otimes \rep(w)^*
		\;=\;
		p^*(\rep).
		\]
		Let \(\Pi(\groupW)\) denote the class of all unitary representations of \(\groupW\) on separable Hilbert space. A polynomial \(p\) is called \df{positive}, written \(p \succeq 0\), if \(p(\rep)\) is positive semidefinite for every \(\rep \in \Pi(\groupW)\). Thus \(  B(\cE) \otimes \C[\groupW]\) becomes an ordered \(*\)-algebra. Moreover, \(p\) is hermitian if and only if \(p(\rep)\) is hermitian for every \(\rep \in \Pi(\groupW)\).
		
		\smallskip

		\smallskip 
		
		The  second main result, Theorem \ref{t:main2}, 
		is a noncommutative Fejér--Riesz theorem and 
		provides a sum-of-squares representation for positive operator-valued trigonometric polynomials on a free product of finite abelian groups.
		It generalizes \cite{KLM} (cf.~\cite{NT13}) and 
		identifies extent as the appropriate notion of complexity for optimal positivity certificates.
		Its proof is given in Section \ref{s:WS} as a corollary of Theorem \ref{t:main1} and Boca's theorem \cite{Boc91}.
		The argument follows an outline 
		generously shared with us by Mehta-Slofstra-Zhao \cite{MSZ} (see also \cite{MSZ23}) 
		adapted to handle  the operator, as opposed to scalar, coefficients appearing here.
		
		\begin{thmA}\label{t:main2}
			Let $\cE$ be any separable (finite or infinite-dimensional) Hilbert space. If $p\in  B(\cE)\otimes \C[\groupW]$ is a polynomial of  extent $\vd,$ then the following are equivalent:
			\begin{enumerate}[\rm (i)]\itemsep=6pt
				\item For any $\tau \in \Pi(\groupW),$ $p (\tau) \succeq 0;$
				\item\label{item:tsos} There exist a positive integer $N$  and polynomials $q_{1}, \ldots, q_{N} \in  B(\cE) \otimes \C[\groupW]$ of extent at most $\lfloor \frac{\vd}{2} \rfloor + 1$ such that 
				\begin{equation}\label{eq:tsos}
					p \ =\ \sum\limits_{i=1}^{N} q_{i}^{*} q_{i}.
				\end{equation}
			\end{enumerate}
		\end{thmA}

		\begin{remark}\rm \label{r:factor}
			The following remarks concern Theorem \ref{t:main2}.
			
			\begin{enumerate}[\rm (a)]\itemsep=5pt
				\item Item \ref{item:tsos} can also be phrased as a factorization result. Letting $q = {\rm col} (q_{1}, \cdots, q_{N}) \in  B(\cE, \cE^{N}) \otimes \C[\groupW],$ \eqref{eq:tsos} simply states 
				\[
				p \ =\ q^{*}q.
				\]
				In particular, if $\cE$ is infinite-dimensional, then $N$ in \eqref{eq:tsos} can be chosen to be one.
				
				\smallskip
				
				In the case that $\cE$ is finite dimensional the bound $N$ on the number of summands in \eqref{eq:tsos}
				can be chosen at most $(\dim \cE)\, (\sum_{j=1}^m |\bG_j| ) \, N(\vd).$  In particular, it depends only on the
				degree of $p,$ the dimension of $\cE$ and the cardinalities of the $\bG_j.$ 
				
				\item
				The proof of Theorem~\ref{t:main2} reduces to the case of finite cyclic groups, $\bG_i = \Z_{n_i}.$ (See Subsection~\ref{ssec:abel2cyc}.)
				In that case there is a natural notion of degree for polynomials based on placing the shortlex order on reduced words:
				given a word $w,$ a polynomial has degree at most $w$ if it is a $ B(\cH)$-combination of words $u^{-1}v$ where 
				$u,v\le w$ 
				and has {\it analytic degree} 
				at most $w$ if it is a linear combination of words of length at most $w.$
				A slightly stronger version of Theorem~\ref{t:main2} was proved in \cite{KLM} for the case where each $\bG_j=\Z_2,$
				in that a priori optimal bounds are obtained.
				Namely, if $p$ has degree at most $w,$ then $p=q^*q,$ for some $q$ with analytic degree at most $w.$ (Here $u^*=u^{-1}$ for a word $u$.)
				Since, in this case, degree at most $w$ implies extent at most $|w|$ (the length of $w$), it follows 
				that if $p$ has extent at most $2d,$ then it factors as $p=q^*q$ where $q$ has extent at most $d.$
				In general, such simple bounds fail; for instance, see \cite[Example 8.1]{KLM} for $\Z_2 * \Z_3$,
				where it is shown, letting $x$ denote the generator of $\Z_2$ and $y$ 
				a generator of $\Z_3,$ there is a polynomial of degree $y$ that does not factor as $q^*q$
				for a $q$ of analytic degree at most $y.$

				\item 
				A scalar-valued variant of Theorem \ref{t:main2} was suggested in \cite[Section 6]{NT13}
				and proved in \cite[Theorem 3.2.1]{GC23}, but to the best of our knowledge, these proofs 
				do not give rise to {\it bounds} in \eqref{eq:tsos} 
				nor do they extend to operator-valued coefficients. 
			\end{enumerate}
		\end{remark}
		
		\subsection{What's new} \label{ss:new}
		\mbox{}\par
		(1)  {Recall that, in Theorem~\ref{t:main1}, the polynomials and the pencil $L$  have coefficients in  $ B({\cH})$ and $ B(\tcH)$
			respectively.} The special case where 
		both $\cH$  and $\tcH$ are finite-dimensional is due to \cite{HKM12}, and its generalization to infinite-dimensional $\tcH$ 
		but still finite-dimensional $\cH$
		is given in \cite{Zal17}.
		While our proof of Theorem \ref{t:main1} follows the now-standard sum-of-squares (sos) strategy, the passage to  coefficients in $ B(\cH)$ with $\cH$ infinite-dimensional introduces new difficulties.\looseness=-1
		
		\smallskip
		
		Suppose that a polynomial {$p\in \cA_{2\vd+1}$} 
		does not belong to the 
		corresponding cone of weighted squares {$\wtC_{2\vd,L}$}. The Hahn--Banach separation theorem yields a linear functional $\varphi$ that is nonnegative on that cone and negative on $p$. 
		In the present 
		operator-valued setting the relevant closedness, and thus separation, is  only available in the product ultraweak topology. 
		To prove this closedness, we use a canonical tuple $A$ arising from truncated left creation operators on Fock space, together with coefficient-extraction estimates and uniform control of Gram-type representations. Further, additional structure arising from completely positive (cp) maps and topologies on spaces of cp maps 
		is utilized.\looseness=-1
		
		\smallskip
		
		{The topological considerations and the need to consider cp maps}  make the representation-theoretic GNS step more delicate: boundedness of the
		resulting representing tuple is no longer automatic.
		To overcome this, we develop a GNS construction adapted  to such
		ultraweakly continuous separating functionals, where the boundedness of the representing tuple is established within the construction itself. We first do this when $\cD_L$ is bounded,
		using the structure of the cone and coefficient-extraction estimates, and then reduce the
		general unbounded case to the bounded one by approximation.
		In this way, the GNS construction produces a
		tuple {$Y\in \cD_L$}  and a representing vector $\gamma$ such that
		\[
		\varphi(p) \ =\  \langle p(Y)\gamma,\gamma\rangle<0,
		\]
		certifying $p$ is not positive on $\cD_L.$

		\smallskip
		
		(2)
		{The second part of the paper employs  Theorem~\ref{t:main1} to establish a ``perfect'' Positivstellensatz 
			for the universal $\ast$-algebra $\mathrm{povm}(n)$ associated with POVMs on $\{1,2,\ldots,n\}$
			(See Theorem \ref{t:posforpovmm}). 
			Since a group $C^*$-algebra $\C[G]$ of a finite abelian group  $G$ depends only on the order $|G|$ of the group $G$, it suffices to prove Theorem \ref{t:main2} for a free product of finite cyclic groups.
			Free product methods}
		and Boca's theorem \cite{Boc91} are then invoked to obtain the desired sum-of-squares representation 
		for positive operator-valued polynomials on free products of finite cyclic groups.
		
		\subsection{Reader's guide}
		
		The paper is structured as follows. Section \ref{sec:prelim} collects the preliminary material used throughout the paper. In particular, we recall some results about completely positive maps, introduce the cones of weighted squares (appearing in item
		\ref{eq:wsos} of Theorem \ref{t:main1}), and review the Fock-space coefficient-extraction machinery from \cite{JKM25}. In Section \ref{sec:top} we introduce the product weak operator and product ultraweak topologies on $\cA_{\ad}$ and prove the closedness of the cones of weighted squares in the product ultraweak topology, see Proposition \ref{prop:closedcone}. The fact that the cone is closed allows for an application of the Hahn--Banach Separation Theorem.
		
		\smallskip
		
		Section \ref{sec:GNS} contains a GNS-type construction. Starting from an ultraweakly continuous linear functional (obtained from the Hahn--Banach separation) that is nonnegative on the cone of weighted squares, we construct a Hilbert space, a self-adjoint operator tuple, and a cyclic vector realizing the functional by evaluation. This construction is the main representation-theoretic ingredient in the proof of Theorem \ref{t:main1}. Section \ref{sec:proofoft1} then proves Theorem~\ref{t:main1} by combining the closedness results of Section \ref{sec:top} with the GNS construction from Section \ref{sec:GNS} and a finite-dimensional compression argument. In Section~\ref{sec:fd}, we explain how Theorem \ref{t:main1} recovers, as special cases, earlier results of Helton--Klep--McCullough \cite{HKM12} and Zalar \cite{Zal17}. 
		
		\smallskip
		
		Section \ref{sec:aff} treats non-monic linear pencils by an affine linear change of variables, and then applies this framework to a special linear pencil that is used later in the proof of Theorem~\ref{t:posforpovmm}. Section \ref{sec:povm} introduces the $*$-algebra $\povm(\underline n)$, develops its free-product structure, and proves a perfect Positivstellensatz for its positive elements. Finally, Section \ref{s:WS} combines this Positivstellensatz and  Boca's theorem  to prove Theorem~\ref{t:main2} for operator-valued polynomials on free products of finite abelian groups.

		\section{Preliminaries}\label{sec:prelim}  
		We combine two classical results about cp maps,
		namely the Arveson Extension Theorem and Stinespring Dilation Theorem, in a form that we will use repeatedly throughout the paper. 
		For a more detailed discussion and for the proofs, we refer the reader to \cite{D25, Pau03}.

		\begin{thm}[Stinespring-Arveson] \label{thm:Arv-Stine}
			If $\oS \subset \oC$ is an operator system contained in a $C^*$-algebra $\oC,$ if 
			$\cE$ is a Hilbert space, and if  $\pi :\oS\to  B(\cE)$ 
			is a cp map, then there is a Hilbert space $\cF,$ a $*$-representation $\tau: C^*(\oS) \to  B(\cF)$ and a bounded linear map $T: \cE\to\cF$ such that $\pi (a) = T^* \tau (a) T,$ where $C^*(\oS)$ is the $C^*$-algebra generated by $\oS.$  Moreover, if $\pi$ is ucp, then $T$ is an isometry; and  if $\oS$ is finite  dimensional and $\cE$ is separable, then $\cF$ 
			can be chosen separable as well.
		\end{thm}
		
		\begin{remark} \rm 
			A separable choice of $\cF$ is possible since  $C^*(\oS)$ and $\cE$ are both separable. In fact, the conclusions of Theorem~\ref{thm:Arv-Stine} hold
			with $\oS$ replaced by any operator system $\oS^\prime\supset \oS$ for which $C^*(\oS^\prime)$ is separable.
		\end{remark}

		\subsection{Convex cone of weighted squares}
		\label{ssec:cones}
		Index $\mathbb{\cH}^{N(\vd)}$ and $\cA_{d}^{N(\vd)}$ (the algebraic direct sum of $\cA_{\vd}$ with itself $N(\vd)$ times) by $\la x \ra_{\vd}.$ 
		For positive integers $\mu,$  let  $V_{\mu} \in \cA_{\vd}^{N}$ denote the \df{Veronese column vector} whose $w\in \la x \ra_{\vd}$ entry is $w$ (adopting the  usual convention of viewing $w$ as the $ B(\cH)$-valued polynomial $I_\cH \, w$).  
		For instance, if $\vg=2$ and $\vd=2$, then  \index{$V_\mu$} 
		\[
		V_2\ =\  \text{col} \begin{pmatrix}
			1 & x_1 & x_2 & x_1^2 & x_1 x_2 & x_2 x_1 & x_2^2
		\end{pmatrix}.
		\]
		Let \df{$\Sigma_\vd$} denote the \df{cone of squares} of polynomials of degree at most $\vd,$
		\begin{equation}\label{eq:sosdef}
			\Sigma_{\vd} \ :=\  \{\, r^{*} r\ : \quad r\in \cA_{\vd}\,\}
			\ \subset \ \cA_{2\vd}.
		\end{equation}
		Given $r\in \cA_\vd,$ the row vector  $R$  with $w$ entry $R_w$ 
		is called the \df{coefficient vector} of $r$ since  $r = R V_{\vd}.$ In particular, 
		\[
		r^* r \ =\  V_\vd^* R^* R  V_\vd
		\]
		so that  $r^*r$ has a representation  as  $V_\vd^* \SG V_\vd$ for a psd block matrix $\SG.$ It was proved in \cite[Proposition 2.2]{JKM25} (see also \cite[Remark 2.4]{JKM25}) that $\Sigma_{\vd}$ is a convex cone. 
		
		\smallskip
		
		We introduce two new cones of weighted squares. Let $\wtC_{\vd,L}$ denote the \df{cone of weighted squares} of polynomials of degree at most $\vd,$
		\begin{equation} \label{eq:wsosdef}
			\wtC_{\vd,L} \ :=\  \{\, r^{*}r + q^{*} \pi(L) q \ :\  r,q \in \cA_{\vd},\ \pi \in \UCP(\oS_{L},  B(\cH))\,\} \ \subset \ \cA_{2\vd+1},
		\end{equation}
		and 
		\begin{equation} \label{eq:wsosdef2}
			\wtC_{\vd+1,\vd,L} \ :=\  \{\, r^{*}r + q^{*} \pi(L) q\ : \  r\in \cA_{\vd+1}, \  q \in \cA_{\vd}, \  \pi \in \UCP (\oS_{L},  B(\cH))\,\} \ \subset\  \cA_{2\vd+2}.
		\end{equation}
		It is clear that $\wtC_{\vd,L} \subset \wtC_{\vd+1,\vd,L}.$ We will suppress the subscript $L$ when it is clear from the context.

		\begin{prop}
			The cones of weighted squares $\wtC_{\vd}$ and $\wtC_{\vd+1,\vd}$ defined in \eqref{eq:wsosdef} and \eqref{eq:wsosdef2}, respectively, are closed under addition. 
		\end{prop}
		
		\begin{proof}
			We establish the result for $\wtC_\vd,$ the case of $\wtC_{\vd+1,\vd}$ being similar.
			
			\smallskip
			
			It is enough to show that 
			\[
			q_{1}^{*} \pi_{1}(L) q_{1} + q_{2}^{*} \pi_{2}(L) q_{2}\ \in\ \wtC_{\vd}
			\]
			for any $q_{1},q_{2} \in \cA_{\vd}$ and $\pi_{1},\pi_{2} \in \UCP(\oS_{L},  B(\cH)).$ 
			Because $\cH$ is infinite  dimensional there is a unitary $U: \cH \oplus \cH \to \cH.$
			The result now follows from the identity, %
			\begin{align*}
				q_{1}^{*} \pi_{1}(L) q_{1} + q_{2}^{*} \pi_{2}(L) q_{2}
				\ =&\ \begin{bmatrix} q_{1}^{*} & q_{2}^{*} \end{bmatrix} \ \begin{bmatrix} \pi_{1} (L) & 0 \\ 0 & \pi_{2}(L)\end{bmatrix} \ \begin{bmatrix} q_{1} \\ q_{2} \end{bmatrix} \\
				\ =&\ \begin{bmatrix} q_{1}^{*} & q_{2}^{*} \end{bmatrix} \ U^{*}U\ \begin{bmatrix} \pi_{1} (L) & 0 \\ 0 & \pi_{2}(L)\end{bmatrix}\ U^{*}U \ \begin{bmatrix} q_{1} \\ q_{2} \end{bmatrix} \\
				\ =&\ q^{*} \pi(L) q,
			\end{align*}
			where $q = U \begin{bmatrix}
				q_{1} \\ q_{2}
			\end{bmatrix} \in \cA_{\vd}$ and $\pi = U \begin{bmatrix}
				\pi_{1} & 0 \\ 0 & \pi_{2}
			\end{bmatrix} U^{*} \in \UCP(\oS_{L},  B(\cH)).$  %
		\end{proof}
		
		\subsection{Fock space and coefficient extraction}\label{ssec:fock}
		This subsection reviews the full Fock space and the creation operators, introduces their symmetrized versions in \eqref{eq:symcreate} (compressed to a suitable finite-dimensional subspace), and summarizes relevant results from \cite[Section 3]{JKM25}.   One difference here is that we
		(implicitly) work with a scaled version of the creation operators. See equation~\eqref{eq:scaleA}.
		
		\smallskip
		
		The full Fock space can be defined over any Hilbert space. The \df{full Fock space} over $\mathbb{C}^{\vg},$ denoted \df{$\cF^2_{\vg}$},  is: 
		\[
		\cF^{2}_{\vg} \ =\  \bigoplus_{n=0}^{\infty} (\mathbb{C}^{\vg})^{\otimes n},
		\]
		where $(\mathbb{C}^{\vg})^{\otimes 0} := \mathbb{C}$ represents the \df{vacuum vector} \df{$\Omega$}.
		Thus elements of $\cF^{2}_{\vg}$ are sequences $(\psi_{0}, \psi_{1}, \psi_{2}, \dots )$ with $\psi_{n} \in (\mathbb{C}^{\vg})^{\otimes n}$ and $\| (\psi_{0}, \psi_{1}, \psi_{2}, \dots ) \|^{2} = \sum\limits_{n=0}^{\infty} \|\psi_{n}\|^{2} < \infty.$
		
		\subsubsection{Left creation operators} 
		Let $\{e_{1}, \dots, e_{\vg}\}$ be any orthonormal basis of $\mathbb{C}^{\vg}.$ With any $w = x_{i_{1}} \dots x_{i_{n}} \in \la x\ra,$ associate a vector 
		\[
		e_{w} \ =\  e_{i_{1}} \otimes \dots \otimes e_{i_{n}} \in (\mathbb{C}^{\vg})^{\otimes n}.
		\]
		The set $\{e_{w}: w \in \la x\ra \}$ forms an orthonormal basis for $\cF^{2}_{\vg}$, with $e_{\varnothing}$ corresponding to the vacuum vector $\Omega.$
		
		\smallskip
		
		For each $j =1, \dots, \vg,$ define the \df{left creation operator}  \df{$C_j$} on $\cF^{2}_{\vg}$ by 
		\begin{equation}\label{eq:create}
			C_{j} (e_{w}) \ =\  e_{x_{j} w} \in (\mathbb{C}^{\vg})^{\otimes (|w|+1)}, \quad (w \in \la x\ra).
		\end{equation}
		Clearly, each $C_{j}$ is an isometry. Moreover $C_{i}^{*} C_{j} = 0 $ if $i \neq j.$ 
		
		\smallskip
		
		Fix a positive integer $\ad\ge 2\vd+2.$ \index{$\ad$}  Let $\Ftgd$ denote the subspace of $\cF^{2}_{\vg}$ spanned by 
		$\{e_{w}: w \in \la x\ra_{\ad} \}$
		and $\iota=\iota_\ad:\Ftgd\to \cF^{2}_{\vg}$ the inclusion. Thus, for instance, for $|w|\le \ad,$
		\begin{equation*}
			\iota^*C_j \iota e_w \ =\  \iota^* C_j e_w \ =\ \begin{cases} e_{x_j w} & \text{ if } |w|<\ad \\
				0 & \text{ if } |w|=\ad. \end{cases}
		\end{equation*}
		Similarly, if $|v|\le \ad,$ then 
		\begin{equation*}
			\iota^* C_j^* \iota e_v \ =\  \iota^*C_j^* e_v \ =\  \begin{cases}  e_u & \text{ if } v=x_j u 
				\\ 0 & \text{otherwise}. \end{cases} 
		\end{equation*}
		Let \df{$A$} $= (A_{1}, \cdots, A_{\vg})$ be defined as
		\begin{equation}
			\label{eq:symcreate}
			A_j \ =\  \iota^* (C_j + C_j^*) \iota.
		\end{equation}
		
		Because
		\[
		L(tA) \ =\ I+t\sum_{j=1}^{\vg}\bA_j\otimes A_j \ \succeq\  \frac12 I
		\]
		for all sufficiently small \(t>0\), a rescaling of the variables \(x\mapsto tx\) allows us to assume, without loss of generality, that
		\begin{equation}
			\label{eq:scaleA}
			L(A)\succeq \frac12 I.
		\end{equation}
		To streamline the exposition, we suppress the corresponding scaling factors in the sequel.
		
		\begin{lemma}[\protect{\cite[Lemma 3.2]{JKM25}}] \label{lem:invertible M}
			The $N(\ad) \times N(\ad)$ scalar matrix $\ME_\ad$ with transpose
			\[
			\ME_{\ad}^\top \ =\  \begin{bmatrix} \langle A^{w} \Omega, e_{v} \rangle\end{bmatrix}_{v,w\in\la x\ra_{\ad}}
			\]
			is invertible.
		\end{lemma}
		
		\subsubsection{Extraction formula for coefficients}\label{ssec:coeff}
		Let $q = \sum Q_{w} w \in \cA_{\ad},$ and let $Q$ be the coefficient row vector of $q.$ For $v\in \la x\ra_{\ad}$ %
		define the linear functional
		\[
		\Omega_{v} :  B(\Ftgd) \to \mathbb{C}; \qquad \Omega_{{v}}(T) \ =\  \langle  T \Omega, e_{v} \rangle. 
		\]
		The operator coefficients \(Q_v\) are obtained from $q(A)$  by solving the linear system
		\begin{align*} 
			Z_{v}(q) \ :=&\ (\mathrm{id}_{ B(\cH)} \otimes \Omega_{v}) q(A) \ =\   \sum\limits_{w} Q_{w} \otimes \Omega_{v} (A^{w}) \\
			\ =&\   \sum\limits_{w} \langle  A^{w} \Omega , e_{v} \rangle \, Q_{w} \ =\   \sum\limits_{w}  [\ME_{\ad}^\top]_{v,w} Q_{w},
		\end{align*}
		where $[\ME_{\ad}^\top]_{v,w}$ is the $(v,w)$ entry of the matrix $\ME_{\ad}^\top.$  In short, \index{$[\ME_d]$}
		\begin{equation}\label{eq:coeff2}
			Z(q) \ =\  Q  \ME_{\ad}, 
		\end{equation}
		where $Z(q)$ and $Q$ are row vectors with $Z_{v}(q)$ and $Q_{v}$  as the $v^{\rm th}$ entry of $Z$ and $Q,$ respectively. 
		Since,  by Lemma~\ref{lem:invertible M}, $\ME_d$ is invertible, 
		\begin{equation} \label{eq:coeff}
			Q \ =\   Z(q) \ME_{\ad}^{-1}.
		\end{equation} 
		{We refer to $\ME$ as the \df{extraction matrix}, and equation~\eqref{eq:coeff} as the \df{extraction formula} for the coefficients of $q.$ Note that 
			the extraction formula depends only upon $q(A);$ that is, the coefficients of $q$ are determined uniquely by $q(A).$}
		
		\smallskip
		
		It follows from equation~\eqref{eq:coeff}  that  there exists a positive constant $\lambda_{\ad}$ (independent of $q$) such that 
		\begin{equation} \label{eq:Coeff bound}
			\|Q_{w}\| \ \leq\  \lambda_{\ad} \, \|q(A)\| \quad \text{ for all $w \in \la x\ra_{\ad}$.}
		\end{equation}
		
		Recall the Veronese column vector $V_\mu$ from the outset of Subsection~\ref{ssec:cones}.  Given a  $p \in \Sigma_{\ad},$ set 
		\begin{equation}
			\label{eq:Gammap}
			\Gamma_p \ =\  \{\, \SG \in  B(\cH)^{N(\ad)\times N(\ad)} \ : \quad \SG \succeq 0, \quad   V_\ad^* \SG V_\ad = p \,\}.
		\end{equation}
		
		\begin{prop}[\protect{\cite[Proposition 3.3]{JKM25}}] \label{prop:bounded}
			For $p \in \Sigma_{\ad}$ the set $\Gamma_p$  is non-empty and 
			norm bounded $($with respect to the operator norm on $ B(\cH^{N(\ad)})$$)$. More precisely, 
			there exists a constant $\mu_{\ad}$ $($depending only on $\ad$ and $\vg$ and not on $p$$)$ such that,
			for all $\SG\in\Gamma_p,$ 
			\[ 
			\|\SG\| \ \leq\  \mu_{\ad} \; \|p(A)\|.
			\]
		\end{prop}
		
		\section{Topologies on \texorpdfstring{$\cA_{\ad}$}{Ad}}\label{sec:top}
		
		In this section we introduce two topologies on $\cA_{\ad}$ used in the sequel and prove that the cone of weighted squares $\wtC_{\vd+1,\vd}$ is closed in the product ultraweak topology,
		defined immediately below.
		
		\smallskip
		
		Identify each polynomial $p=\sum_{w\in \la x\ra_{\ad}} P_w w\in \cA_{\ad}$ with its coefficient tuple $(P_w)_{w\in \la x\ra_{\ad}}\in  B(\cH)^{\la x\ra_{\ad}}$. We equip $\cA_{\ad}$ with the \df{product weak operator topology (WOT)} and the \df{product ultraweak topology} inherited from $ B(\cH)^{\la x\ra_{\ad}}$. Thus a net of polynomials
		\[
		p_\alpha=\sum_{w\in \la x\ra_{\ad}} P_{\alpha,w}w
		\]
		converges to a polynomial $p=\sum_w P_w w$ if and only if $P_{\alpha,w}\to P_w$ for each $w\in \la x\ra_{\ad}$ in the WOT, respectively in the ultraweak topology. In either case, $\cA_{\ad}$ is a locally convex topological vector space.
		
		\subsection{Point-WOT}
		To pass to limits of completely positive maps, we use the following topology. Let $\mathfrak{E}$ be a closed subspace of a $C^*$-algebra and $\cK$ a Hilbert space. A net $(\pi_{\alpha})_{\alpha}$ in $ B(\mathfrak{E},  B(\cK))$ converges to $\pi$ in the \df{point-WOT topology} if $\pi_{\alpha}(a)\to \pi(a)$ in the WOT for all $a\in \mathfrak{E}$. Complete positivity is preserved under point-WOT limits. The closed unit ball of $ B(\mathfrak{E}, B(\cK))$ is compact Hausdorff in the point-WOT; see \cite[Definition 14.7.6]{D25}.
		
		\subsection{Closedness of the cone \texorpdfstring{$\wtC_{\vd+1,\vd}$}{Cd+1,d}}
		
		Recall that $A$ has been scaled so that $L(A)\succeq \frac12$; see \eqref{eq:scaleA}.
		
		\smallskip
		
		Let $\cT(\cH)$ denote the trace-class operators on $\cH$. The space $\cT(\cH)^{\la x\ra_{\ad}}$, equipped with the norm
		\[
		\|(T_w)_w\|_{1} \ =\ \sum_{w\in \la x\ra_{\ad}} \|T_w\|_{1},
		\]
		is a Banach space. Its dual is $ B(\cH)^{\la x\ra_{\ad}}$ with norm
		\[
		\|(B_w)_w\| \ =\ \max_{w\in \la x\ra_{\ad}} \|B_w\|.
		\]
		Under this identification, the weak-$*$ topology induced by $\cT(\cH)^{\la x\ra_{\ad}}$ coincides with the product ultraweak topology. Accordingly, for
		\[
		p=\sum_{w\in \la x\ra_{\ad}} P_w w \in \cA_{\ad},
		\]
		we define
		\[
		\|p\|:=\max_{w\in \la x\ra_{\ad}} \|P_w\|.
		\]
		This norm is relevant only for this subsection. 
		
		\begin{prop} \label{prop:closedness of bounded cone}
			For any $t>0$, the set
			\[
			\wtC_{\vd+1,\vd,t} \ :=\  \{\, p \in \wtC_{\vd+1,\vd} : \|p\|\le t \}
			\]
			is closed in the product ultraweak topology on $\cA_{\ad}$. The same holds with $\wtC_{\vd+1,\vd,t}$ replaced by $\wtC_{\vd,t}$.
		\end{prop}
		
		The proof uses the following two lemmas.
		
		\begin{lemma}
			\label{l:cp-to-ucp}
			Let $\cH$ be a fixed separable infinite-dimensional Hilbert space. If $\oS$ is a finite-dimensional unital  operator system, $\cE $ is a separable (infinite-dimensional) Hilbert space and $\psi:\oS\to  B(\cE )$ is completely positive (cp), then there is a ucp map $\pi: \oS \to  B(\cH)$ and a bounded operator $V :\cE\to\cH $ such that
			\[
			\psi(X) \ =\  V^* \pi(X) V
			\]
			for all $X\in\oS.$
		\end{lemma}
		
		\begin{proof}
			By Theorem \ref{thm:Arv-Stine}, there exists a separable (infinite-dimensional) Hilbert space $\cF,$ a $*$-representation $\tau: C^*(\oS) \to  B(\cF),$ and a bounded operator $\VT: \cE \to \cF$ such that $\psi(X) = \VT^{*} \tau(X) \VT.$ Since $\cH$ and $\cF$ are both separable and infinite  dimensional, 
			there is a unitary operator $U: \cF \to \cH.$ The map $\pi: \oS \to  B(\cH)$ defined by
			\[
			\pi(X) \ = \ U \tau(X) U^{*}
			\]
			is cp. Setting $V = U \VT $ gives, 
			\[
			\psi(X) \ =\ \VT^{*} \tau(X) \VT \ = \ \VT^{*} U^{*} \pi(X) U \VT \ =\  V^*  \pi(X) V.  \qedhere
			\]
		\end{proof}

		\begin{lemma}\label{lem:pre-bounded-cone}
			Let $\oS=\spann\{I,B_1,\dots,B_\vh\}$ be a finite-dimensional operator system with $B_j$ self-adjoint and $\vh\ge \vg$. Let $(\penM_\alpha)_\alpha$ be a net of monic linear pencils
			\[
			\penM_\alpha(x) \ =\ I+\sum_{j=1}^{\vg} D_{\alpha,j} x_j,
			\]
			such that $\penM_\alpha(A)\succeq \frac12$ and $D_{\alpha,j}\in \spann\{B_1,\dots,B_\vh\}$. Suppose $D_{\alpha,j}\to D_j$ in norm for each $j$, and set
			\[
			\penM(x) \ =\ I+\sum_{j=1}^{\vg} D_j x_j.
			\]
			
			Let $r_\alpha\in \Sigma_{\vd+1}$, $q_\alpha\in \cA_\vd$, and $\pi_\alpha\in \UCP(\oS, B(\cH))$, and define
			\[
			p_\alpha \ =\  r_\alpha + q_\alpha^* \pi_\alpha(\penM_\alpha) q_\alpha.
			\]
			
			If there exists $\kappa>0$ such that $\|p_\alpha\|\le \kappa$ for all $\alpha$ and $(p_\alpha)_\alpha$ converges to $p$ in the product ultraweak topology, then there exist $r\in \Sigma_{\vd+1}$, $q\in \cA_\vd$, and $\pi\in \UCP(\oS, B(\cH))$ such that
			\[
			p \ =\  r + q^* \pi(\penM) q.
			\]
			
			The same conclusion holds with $\Sigma_{\vd+1}$ replaced by $\Sigma_\vd$.
		\end{lemma}
		
		\begin{proof}
			Recall the tuple $A$ from equation~\eqref{eq:symcreate} 
			derived from the creation operators and scaled so that 
			$L(A)\succeq \frac{1}{2}.$
			Choose $K\ge \kappa \sum_{w\in \la x\ra_{2\vd+2}} \|A^w\|$ and note,
			\[
			\|p_\alpha(A)\|\ \le\  \|p_\alpha\|\sum_{w\in \la x\ra_{2\vd+2}} \|A^w\| \ \le\  K.
			\]
			Since $r_\alpha\in \Sigma_{\vd+1}$ and $\pi_\alpha(\penM_\alpha)(A)\succeq \frac12$,
			\[
			0 \ \preceq\  r_\alpha(A), \qquad \frac12 q_\alpha^*q_\alpha(A) \ \preceq\  p_\alpha(A),
			\]
			and hence $\|r_\alpha(A)\|,\|q_\alpha(A)\|\le K$.
			
			\smallskip
			
			Let $G_\alpha\in \Gamma_{r_\alpha}$. By Proposition~\ref{prop:bounded}, the family $(\SG_\alpha)_\alpha$ is uniformly bounded. By Banach--Alaoglu, there exists a subnet $(\SG_\beta)_\beta$ converging ultraweakly to some $\SG\succeq 0$. Set $r=V_{\vd+1}^*\SG V_{\vd+1},$ where $V_{\vd+1}$ is the Veronese column vector from Subsection~\ref{ssec:cones}.  Using \eqref{eq:coeff}, evaluation at $A$ determines coefficients, hence $r_\beta\to r$ in the product ultraweak topology.
			
			\smallskip
			
			Let $Q_\beta$ denote the coefficient row vector of $q_\beta$ so that $q_{\beta} = Q_{\beta} V_{\vd}$. The maps 
			\[
			\psi_{\beta} : \oS \to  B(\cH^{N(\vd)}); \qquad X \ \mapsto \ Q_\beta^*\pi_\beta(X)Q_\beta
			\]
			are cp. By equations~\eqref{eq:Coeff bound} and the uniform bound on $\|q_\alpha(A)\|,$ 
			\[
			\sup\limits_{\beta} \| Q_{\beta} \| \ <\  \infty.
			\]
			Thus the net of cp maps $(\psi_{\beta})_{\beta}$ is uniformly bounded in the operator 
			norm. Therefore, there exists a subnet $(\psi_{\gamma})_{\gamma} $ of $(\psi_{\beta})_{\beta} $
			that converges to some cp map say $\psi: \oS \to  B(\cH^{N(\vd)})$ in the point-WOT. By Lemma \ref{l:cp-to-ucp}, there exists a ucp map $\pi : \oS \to  B(\cH)$ and a bounded operator $Q: \cH^{N(\vd)} \to \cH$ such that 
			\[
			\psi(X) \ = \ Q^{*} \pi(X) Q.
			\]
			
			\smallskip
			
			Since $(\psi_{\gamma})_{\gamma}$ converges to $\psi$ in the point-WOT,
			the net  
			$\psi_{\gamma}(I_{\tcH})=Q^*_\gamma Q_\gamma$  converges to $\psi(I_{\tcH})=Q^*Q$ in the WOT.
			Likewise,   {since also  the nets $(D_{\gamma,j})_\gamma$ norm converge to $D_j$ and the maps
				$\psi_\gamma$ are uniformly norm bounded}, 
			the net  $\psi_{\gamma}(D_{\gamma,j})$ converges to $\psi(D_j)$ in the WOT for each $j.$
			Hence, from the identity, 
			\[
			\begin{split}
				q_{\gamma}^{*} \pi_{\gamma}(\penM_\gamma) q_{\gamma} 
				\ =&\ V_{\vd}^{*} Q_{\gamma}^*  \Big( I_{\cH} + \sum\limits_{j=1}^{\vg} \pi_{\gamma} (D_{\gamma,j}) x_{j}\Big) Q_{\gamma} V_{\vd} \\
				=&\ V_{\vd}^{*} Q_{\gamma}^{*}Q_{\gamma} V_{\vd} + V_{\vd}^{*} \Big(\sum\limits_{j=1}^{\vg} \psi_\gamma (D_{\gamma,j}) x_{j} \Big) V_{\vd}, \\ 
			\end{split}
			\]
			it follows that $q_\gamma^* \pi_\gamma(\penM_\gamma) q_\gamma$ converges in the product WOT on $\cA_{\ad}$ to 
			\[
			\begin{split}
				\ \  V_{\vd}^* Q^*Q V_{\vd} + V_{\vd}^* \Big(\sum\limits_{j=1}^{\vg} \psi (D_{j}) x_j 
				\Big ) V_{\vd} 
				\ \ & =  \  V_{\vd}^{*} Q^* Q V_{\vd} + V_{\vd}^{*} Q^* \Big(\sum\limits_{j=1}^{\vg} \pi (D_{j})  x_{j} \Big) Q V_{\vd} 
				\\ & = \ V_{\vd}^{*} Q^* \pi(\penM) Q  V_{\vd} 
				\ =\ q^{*} \pi(\penM) q,
			\end{split}
			\]
			where $q = Q V_{\vd}\in \cA_{\vd}.$  
			Hence  $p_\gamma = r_{\gamma} + q_{\gamma}^{*} \pi(\penM_\gamma) q_{\gamma}$ 
			converges to $r + q^{*} \pi(\penM) q \in \wtC_{\vd+1,\vd}$ in the product WOT on $\cA_{\ad}.$ Since, by assumption, 
			$(p_\alpha)$ also converges in the product ultraweak topology to $p,$ it follows that
			$p=r+q^* \pi(\penM)q$ as desired. 
		\end{proof}
		
		\begin{proof}[Proof of Proposition~\ref{prop:closedness of bounded cone}]
			Let $p_\alpha\in \wtC_{\vd+1,\vd,t}$ converge ultraweakly to $p$. Since $\cA_{\ad}$ is a dual space, the norm is weak-$*$ lower semicontinuous, hence $\|p\|\le t$. It remains to show $p\in \wtC_{\vd+1,\vd}$.
			
			\smallskip
			
			Write $p_{\alpha} = r_{\alpha} + q_{\alpha}^{*} \pi_{\alpha}(L) q_{\alpha}$ for some $r_{\alpha} \in \Sigma_{\vd+1},$ $q_{\alpha} \in \cA_{\vd},$ and $\pi_{\alpha} \in \UCP (\oS_{L},  B(\cH)).$ 
			Choosing, in Lemma~\ref{lem:pre-bounded-cone},  $\oS=\oS_L,$ each $\penM_\alpha =L$ and each $D_{\alpha,j} =\bA_j,$  the net
			$(p_\alpha)$ satisfies the hypotheses of that lemma with $\penM=L.$ Hence, there exists $r\in \Sigma_{\vd+1},$ $q\in \cA_\vd$ and  $\pi\in\UCP(\oS_L, B(\cH))$ such that 
			\[
			p = r + q^* \pi(L)q \ \in\  \wtC_{\vd+1,\vd}. \qedhere
			\]
			
			{Since $\cA_{2\vd+1}$ is closed 
				in $\cA_{\ad}$
				in the product ultraweak topology, the result for
				$\wtC_{\vd,t}$ follows by noting that 
				\[
				\wtC_{\vd,t} \ =\ \wtC_{\vd+1,\vd,t}\cap\cA_{2\vd+1},
				\]
				see Lemma \ref{lem:reduction}.} 
		\end{proof}
		
		\begin{prop}\label{prop:closedcone}
			The cone $\wtC_{\vd+1,\vd}$ is closed in the product ultraweak topology on $\cA_{\ad}$. The same holds for $\wtC_\vd$.
		\end{prop}
		
		\begin{proof}
			Combine  Proposition~\ref{prop:closedness of bounded cone} and the Krein--Smulian theorem \cite[Theorem 3.6.2]{D25}.
		\end{proof}

		\section{GNS construction}\label{sec:GNS}
		This section is devoted to the proof of the GNS-inspired result Theorem \ref{thm:GNS} below.

		\begin{thm} \label{thm:GNS}
			Suppose $\cD_{L}$ is bounded. If $\varphi:\cA_{2\vd+2}\to\mathbb C$ is a continuous (in the product ultraweak topology) linear functional such that, 
			for all $q\in \wtC_{\vd+1,\vd},$
			\[
			\varphi(q)\ \geq\  0, %
			\] 
			then there exist a separable Hilbert space $\cE$, a  $\vg$-tuple $Y=(Y_1,\dots,Y_\vg)$ of
			bounded self-adjoint operators on $\cE$, and a vector $\gamma\in \cH\otimes \cE$ such that $L(Y) \succeq 0,$ and 
			\begin{equation} \label{eq:representation formula}
				\varphi(q^ *r)\ =\ \big\langle r(Y)\gamma,\ q(Y)\gamma\big\rangle_{\cH\otimes \cE} \qquad \text{ for all $r \in \cA_{\vd+1}$ and $q\in\cA_{\vd}.$} 
			\end{equation}
			Therefore, for all $p\in\cA_{2\vd+1}$,
			\[
			\varphi(p)\ =\ \langle p(Y)\gamma,\ \gamma\rangle. 
			\]
		\end{thm}

		\begin{remark}\rm 
			While the construction of the auxiliary Hilbert spaces and coordinate maps in Subsections \ref{ssec:4.1} and \ref{ssec:4.2} proceeds without topological assumptions on $\cD _L$, the boundedness of $\cD _L$ is strictly necessary to ensure the boundedness of the left-multiplication operators $Y_j$ defined in Subsection \ref{ssec:E}. More precisely, this assumption is used in the Subsubsection \ref{Ybdd}.
		\end{remark}

		We begin with a brief outline of the argument. We first encode the linear functional $\varphi$ by an 
		$N(\vd+1)\times N(\vd+1)$ positive trace-class block matrix
		$ S=[S_{u,v}]_{u,v\in \la x\ra_{\vd+1}},$ which 
		induces a positive sesquilinear form on the vector space 
		$ \notV=\bigoplus_{w\in \la x\ra_{\vd+1}} \cH. $
		Passing to the quotient by the null space and completing, we obtain an auxiliary Hilbert space
		$\cM.$
		
		\smallskip
		
		For each word $w\in \la x\ra_{\vd+1}$, we consider the ``coordinate'' map $\Phi(w):\cH\to \cM$ 
		that 
		places a vector first in the $w$-th component in $\notV$ followed by the canonical projection $\notV\to\cM$. 
		The Hilbert space $\cE\subseteq \cM$ is then obtained as the closure of the subspace generated
		by the ranges of the coordinate maps corresponding to words of degree at most $\vd$.  We then
		define operators $Y_1,\dots,Y_\vg$ on $\cE$ via left-multiplication maps. Using the positivity of $\varphi$ on $\wtC_{\vd+1,\vd}$, we show that each $Y_j$ is
		well-defined, bounded, and self-adjoint.
		
		\smallskip
		
		Next, using the coordinate map corresponding to the empty word, we construct a vector
		$\gamma\in \cH\otimes \cE$ that satisfies \eqref{eq:representation formula}.  Moreover, the set $ \{\,q(Y)\gamma : q\in \cA_\vd\,\} $
		is dense in $\cH\otimes \cE$. Finally, using this density together with the assumption that
		$\varphi$ is nonnegative on $\wtC_{\vd+1,\vd}$, we show, by testing against ucp maps, that
		$L(Y)\ge 0$.  
		
		\smallskip
		
		We now carry out this construction in detail. The proof proceeds in five steps.
		
		\subsection{The positive block matrix \texorpdfstring{$S$}{S} and an auxiliary Hilbert space \texorpdfstring{$\cM$}{M}}\label{ssec:4.1}
		We begin by encoding the functional $\varphi$ into a positive block operator matrix $S$ and using $S$ to construct a Hilbert space.
		
		\smallskip
		
		Since $\varphi$ is ultraweak continuous, there exist trace class operators $S_{w}$ $(w \in \la x\ra_{2\vd+2})$ in $ B(\cH)$ such that 
		\[
		\varphi(p) \ =\  \sum\limits_{w \in \la x\ra_{2\vd+2}} {\rm Tr}\, (S_{w} P_{w}),
		\]
		where $p = \sum_{w \in \la x \ra_{2\vd+2}} P_{w} w.$ 
		Denote by $S$ the $N(\vd+1) \times N(\vd+1)$ block matrix whose $(v,w)$ entry is $S_{w^{*}v}.$
		For  $r = \sum_{v \in \la x \ra_{\vd+1}} R_{v} v$
		and $r^\prime = \sum_{w \in \la x \ra_{\vd+1}} R^\prime_{w} w,$ we have
		\begin{equation} \label{eq:block-trace}
			\varphi(r^{*}r^\prime) \ =\ \sum\limits_{v,w\in \la x \ra_{\vd+1}} {\rm Tr}\, (S_{v^{*}w} R_{v}^{*}R^\prime_{w}).
		\end{equation}
		Letting $R$ denote the  row operator $R : \cH^{N(\vd+1)} \to \cH,$ with $R_{u}$ as the $u$-th element,
		gives
		\[
		\varphi(r^{*}r) 
		\ =\ \sum\limits_{v,w\in \la x \ra_{\vd+1}} {\rm Tr}\, (S_{v^{*}w} R_{v}^{*}R_{w}) 
		\ =\  \sum_{v,w\in \la x \ra_{\vd+1}} {\rm Tr}\, (\, [S]_{w,v} [R^{*}R]_{v,w}\,) 
		\ =\  \ \ \ {\rm Tr}\, (SR^{*}R).  
		\]
		Given a psd operator $T \in B (\cH^{N(\vd+1)}),$ factor  $T = R^{*}R$ for some $R: \cH^{N(\vd+1)} \to \cH$ (using $\cH$ is infinite-dimensional), let  $r = \sum R_{u} u,$ where $R_{u}$ is the $u^{\rm th}$ element of the row operator $R$ and note
		\[ 
		{\rm Tr}\, (ST) \ =\  {\rm Tr}\, (SR^*R) \ =\  \varphi (r^{*}r) \ \geq\  0.
		\]
		It follows that $S \succeq 0.$ 
		
		\smallskip
		
		We now use the operator $S$ to define a Hilbert space via a GNS-type construction. Consider the vector space 
		\[
		\notV =\ \bigoplus_{w\in \la x \ra_{\vd+1}} \cH.
		\]
		Equip $\notV$ with the sesquilinear form 
		\begin{align*}	
			\langle (\xi_w)_{w},\,(\eta_v)_{v}\rangle_{\notV}\  := &\ \la S (\xi_{w})_{w}, (\eta_{v})_{v} \ra_{\cH^{N(\vd+1)}} 
			\ =\  \sum\limits_{v \in \la x\ra_{\vd+1}} \Big\la \sum\limits_{w\in \la x \ra_{\vd+1}} [S]_{v,w} \xi_{w},\, \eta_{v}\Big\ra_{\cH} \\
			\ =&\  \sum\limits_{v,w \in \la x \ra_{\vd+1}} \la S_{w^{*}v} \xi _{w}, \eta_{v}\ra_{\cH} 
			\ =\  \sum_{v,w\in \la x \ra_{\vd+1}}\!\langle \xi_w,\,S_{v^ *w}\,\eta_v\rangle_\cH . 
		\end{align*}
		This form is psd by the positivity of $S.$ Let 
		\[
		\cN\ =\ \{z\in \notV:\langle z,z \rangle_{\notV}=0\}
		\]
		denote its subspace of null vectors, and let $\cM$ denote the Hilbert space obtained by the completion of the quotient space $ \notV/\cN. $ Clearly, $\cM$ is separable. 
		Let $\rho: \notV\to \cM$  denote the quotient map (followed by the inclusion of $\notV/\cN$ into $\cM$).
		
		\subsection{The coordinate maps \texorpdfstring{$\Phi(w)$}{P} and the Hilbert space \texorpdfstring{$\cE$}{E}}\label{ssec:4.2}
		We now introduce coordinate maps that allow us to identify coefficient vectors inside the Hilbert space $\cM$ and use them to construct the Hilbert space $\cE$ as a subspace of $\cM.$
		
		\smallskip
		
		For each word $w\in \la x\ra_{\vd+1},$ define a linear map 
		\[
		\Phi(w):\cH\to \notV; \qquad \xi\ \mapsto \ \big[(\delta_{v,w}\xi)_{v\in \la x\ra_{\vd+1}}\,\big], 
		\]
		where $\delta_{v,w}$ denotes the Kronecker delta. Thus, for all $v,w\in \la x\ra_{\vd+1}$ and $\xi,\eta\in \cH$,
		\begin{equation}\label{eq:Phi-kernel}
			\big\langle \Phi(w)\xi,\ \Phi(v)\eta\big\rangle_\notV
			\ =\ \langle \xi,\ S_{v^ *w}\,\eta\rangle_\cH .
		\end{equation}
		For later use, observe, identifying $\Phi (w)$ with $\rho \circ\Phi (w),$ 
		equation~\eqref{eq:Phi-kernel}
		gives  $\Phi(w)^* \Phi(w)=S_{w^*w}.$  Since $S_{w^*w}$ is trace class, 
		each $\Phi(w)$ is a Hilbert--Schmidt operator and moreover,
		\begin{equation}
			\label{eq:Swv}
			\Phi(w)^* \Phi(v) \ =\  S_{v^* w}.
		\end{equation}
		
		We define the Hilbert space $\cE$ as 
		\[
		\cE \ :=\  \overline{\mathrm{span}} \big\{\, (\rho \circ \Phi(w))(h) :\  |w| \le \vd,\ h \in \cH \,\big\} \ \subseteq\  \cM.
		\]
		The Hilbert space $\cE$ will serve as the space on which the operators $Y_1,\dots,Y_\vg$ act.
		
		\subsection{The operator tuple \texorpdfstring{$Y$}{Y}}\label{ssec:E}
		We now define operators corresponding to the noncommuting variables and show that they act as bounded self-adjoint operators on $\cE$.
		
		\smallskip
		
		We consider a subspace of the vector space $\notV.$ Let 
		\[ 
		\Ddom \ :=\ {\rm span}\{\, \Phi(w)\xi:\ w\in \II{\vd},\ \xi\in \cH \}\ \subset\  \notV.
		\]
		Note that the Hilbert space $\cE$ is the closure of $\rho(\Ddom)$ in $\cM.$ For each $ j = 1, \cdots, \vg,$ define a linear map $ \Delta_j:\Ddom\to \notV$ by 
		\[
		\Delta_j(\Phi(w)\zeta) = \Phi(x_j w) \zeta.
		\]
		We first show that $\Delta_j$ behaves well with the null vectors in $\Ddom \cap \cN$.
		Suppose $f=\sum_{w\in \II{\vd}}\Phi(w)\xi_w \in \Ddom\cap \cN.$ %
		Given $v\in \la x\ra_\vd$ and $\eta\in \cH$ an application of
		equation~\eqref{eq:Swv} gives, %
		\[
		\big\langle \Delta_j(f)  ,\ \Phi(v)\eta\big\rangle_\notV \ =\  \sum_{w \in \la x\ra_{\vd}}\langle \xi_w,\ S_{v^ * x_j w}\eta\rangle_\cH 
		\ =\   \sum_{w \in \la x \ra_{\vd}}\langle \xi_w,\ S_{(x_jv)^ * w}\eta\rangle_\cH
		\ =\  \big\langle f,\ \Phi(x_jv)\eta\big\rangle_\notV 
		\ =\  0.
		\]
		It follows that $ \langle \Delta_j f, g\rangle_{\notV} \ =\ 0 $ for $f\in \Ddom\cap \cN$ and $g\in \Ddom.$ 
		
		\smallskip
		
		Let $P_\cE$ denote the projection of
		$\cM$ onto $\cE.$ The computation above says that if $f\in \Ddom\cap \cN,$ then $P_\cE \rho (\Delta_j(f))=0.$
		Hence, for each $ j = 1, \cdots, \vg,$ we obtain a linear map
		$\widehat{\Delta}_j:\Ddom\to \cE$ defined by 
		\begin{equation}
			\label{eq:hatLambda}
			\widehat{\Delta}_j f = P_\cE \rho(\Delta_j f)
		\end{equation}
		that maps $\Ddom\cap \cN$ to $0.$ Finally, for each $j=1, \cdots, \vg,$ we define a linear map
		$Y_j:\Ddom/(\cN\cap \Ddom) \to \cE$ by
		\[
		Y_j \rho(f) = P_\cE \rho(\Delta_j(f)),
		\]
		for $f\in \Ddom.$
		
		\subsubsection{The \texorpdfstring{$Y_j$}{Yj} are bounded} \label{Ybdd}
		The boundedness of the $Y_j$ rests on the following fact. 
		
		\begin{lemma} \label{lem:bounded}
			If $\cD _{L}$ is bounded, then there exists a constant $c>0$ such that $c\pm x_j \in \wtC_\vd$ for any $\vd \geq 0.$ Moreover, the constant $c$ does not depend on $\vd.$
		\end{lemma}

		\begin{proof}
			Since $\cD_{L}$ is bounded, there exists a constant $c>0$  such that 
			$\sup\limits_{X \in \cD_{L}}\|X_{j} \| \leq c $
			for all $j = 1,\dots, \vg.$ Fix $j \in \{ 1, \dots, \vg\}.$ Let 
			\[
			L_{j} (x) \ =\ I_{\cH} - \frac{1}{c} x_{j}.
			\]
			Since $\cD_{L} \subseteq \cD_{L_{j}},$ an application of  \cite[Theorem 5.13]{DDSS17} (see also \cite[Theorem 1.1]{Zal17}), produces 
			a ucp map $\pi_{j} : \oS_{L} \to \oS_{L_{j}} \subset  B(\cH)$ such that 
			\[
			\pi_{j} (\bA_{i}) \ =\ -\, \delta_{i,j} \frac{1}{c} I_{\cH},
			\]
			where $\delta_{i,j}$ is the Kronecker delta.   It follows that $\pi_j(L) = I-\frac{1}{c}x_j \in \wtC_\vd$ as desired.
		\end{proof}
		
		We now show that $Y_j$ is bounded. For $f \in \Ddom,$ we have
		\[
		\la Y_{j} \rho(f) , \rho(f) \ra_{\cE} \ \leq \ c\, \la \rho(f), \rho(f)\ra_{\cE},
		\]
		with  $c$ as in Lemma~\ref{lem:bounded}.
		Indeed, for $f = \sum\limits_{w \in \la x\ra_{\vd}} \Phi(w) \xi_{w} \in \Ddom,$ %
		\begin{align*}
			\Big\la Y_{j}\, \rho \Big(\sum\limits_{w \in \la x\ra_{\vd}} \Phi(w) \xi_{w}\Big), & \, \rho\Big(\sum\limits_{w \in \la x\ra_{\vd}} \Phi(w) \xi_{w}\Big) \Big\ra_{\cE}
			\ = \ \sum\limits_{v,w \in \la x\ra_{\vd}} \la\, \rho (\Phi(x_{j}w) \xi_{w}), \rho (\Phi(v) \xi_{v}) \ra_{\cM}  \\[3pt]
			\ = &\ \sum\limits_{v,w \in \la x\ra_{\vd}} \la  \Phi(x_{j}w) \xi_{w}, \Phi(v) \xi_{v} \ra_{\notV} 
			\ = \ \sum\limits_{v,w \in \la x \ra_{\vd}}\la \xi_{w}, S_{v^{*}x_{j}w} \xi_{v} \ra_{\cH} \\[3pt]
			\ = &\ \sum\limits_{v,w \in \la x \ra_{\vd}}\la S_{w^{*}x_{j} v}\xi_{w},  \xi_{v} \ra_{\cH}
			\ = \ \sum\limits_{v,w \in \la x \ra_{\vd}} \operatorname{Tr}\, (S_{w^{*} x_{j} v} \, \xi_{w} \xi_{v}^{*}\,)  \\[3pt]
			\ = & \ \sum\limits_{v,w \in \la x \ra_{\vd}}  \varphi\, ( \xi_{w} \xi_{v}^{*}\,w^{*}x_{j}v)
			\ = \  \varphi\, \Big( \sum\limits_{v,w \in \la x \ra_{\vd}} \, \xi_{w} \xi_{v}^{*}\,w^{*} x_{j} v \Big)  \\[3pt]
			\ = &\ \ \varphi\, \Big( \sum\limits_{v,w \in \la x \ra_{\vd}} R_{w}^{*}R_{v} w^{*}x_{j} v \Big) 
			\ = \ \varphi(r^{*}x_{j}r),
		\end{align*}
		where $r = \sum_{w \in \la x \ra_{\vd}} R_{w} w,$ and $R_{w}$ is the {rank-one operator} 
		that  maps $h \in \cH$ to {$\la h , \xi_{w}\ra\, e$  for some fixed unit vector $e$ in $\cH$.}  
		An application of Lemma~\ref{lem:bounded} (with $c-x_j \in \wtC_\vd$)  gives,
		\[
		\Big\la Y_{j}\, \rho \Big(\sum\limits_{w \in \la x\ra_{\vd}} \Phi(w) \xi_{w}\Big),\, \rho \Big(\sum\limits_{w \in \la x\ra_{\vd}} \Phi(w) \xi_{w}\Big) \Big\ra \ =\ \varphi (r^{*} x_{j} r) \ \leq\  c\, \varphi (r^{*}r).
		\]
		Essentially the same calculation also gives,
		\[
		\varphi(r^{*}r) \ =\  	\Big\| \rho \Big(\sum\limits_{w \in \la x\ra_{\vd}} \Phi(w) \xi_{w}\Big) \Big\|^{2} .
		\]
		Hence  $\la Y_{j} \rho (f) , \rho(f) \ra_{\cE} \leq c\, \la \rho(f), \rho(f)\ra_{\cE}$ for all $f \in \Ddom.$ 
		The same argument (with instead $c+x_j\in\wtC_\vd$) also shows $-c\, \la f,f\ra_\cE \le \la Y_{j} f ,f \ra_{\cE}.$
		Since $\langle Y_j f,f\rangle$ is real for all $f,$ it follows from polarization that $Y_{j}$ is self-adjoint. Because
		$\vert \langle Y_j f,f\rangle \vert \le  c\|f\|$ for all $f$ and $Y_j$ is self-adjoint, $Y_j$ is bounded with $\|Y_j\|\le c.$
		Finally, $Y_j$ extends to a bounded self-adjoint operator on $\cE.$
		
		\smallskip
		
		We now give the details of the polarization argument sketched above.
		For $f,g \in \Ddom,$ by the polarization identity for sesquilinear forms we get
		\[
		\la Y_{j} \rho(f), \rho(g) \ra \ =\ \frac{1}{4}\sum\limits_{k=0}^{3} \iota^{k} \big\la\, Y_{j} \rho (f+\iota^{k}g), \rho (f+\iota^{k} g)  \,\big\ra.
		\]
		Thus,
		\[
		|\,\la Y_{j} \rho(f), \rho(g) \ra\,| \ \leq\  \frac{c}{4} \sum\limits_{k=0}^{3} \| \rho( f+\iota^{k}g) \|^{2} \ =\ c\, (\|\rho(f)\|^{2} + \|\rho(g)\|^{2}). 
		\]
		We claim that $|\,\la Y_{j} \rho( f), \rho(g) \ra\,| \leq 2 c \, \|\rho(f)\|\, \|\rho(g)\|$ for all $f,g \in \Ddom.$ If $\rho(f) = 0,$ then this is trivially true. Assume $\rho(f) \neq 0.$ For any real number $t > 0,$ we have 
		\[
		|\,\la Y_{j} \rho(f), \frac{1}{t} \rho(g) \ra\,| \ \leq\   c (\|\rho(f)\|^{2} + \frac{1}{t^{2}}\|\rho(g)\|^{2}). 
		\]
		This implies that 
		\[
		|\,\la Y_{j} \rho(f), \rho(g) \ra\, | \ \leq\  c (t \|\rho(f)\|^{2} + \frac{1}{t} \|\rho(g)\|^{2})
		\]
		for all $t >0.$ By the AM-GM inequality, we get that  
		\[
		t \|\rho(f)\|^{2} + \frac{1}{t} \|\rho(g)\|^{2}  \leq 2 \, \|\rho(f)\|\, \|\rho(g)\|,
		\]
		where equality holds when $t = \|\rho(g)\| /\|\rho(f)\|.$ Finally, we get that 
		\[
		|\,\la Y_{j} \rho(f), \rho(g) \ra\, | \ \leq\ 2 c \, \|\rho(f)\|\, \|\rho(g)\|
		\]
		for all $f,g \in \Ddom.$ This proves that $Y_{j}$ are bounded on $\Ddom/(\cN\cap \Ddom),$ and hence, can be extended to bounded operators on $\cE.$ 
		
		\subsubsection{The \texorpdfstring{$Y_j$}{Y} are self-adjoint} 
		If  $v,w\in \II{\vd}$ and $\xi,\eta\in \cH,$  then 
		\begin{align*}
			\big\langle Y_{j} \rho(\Phi(w)\xi),\rho(\Phi(v)\eta)\big\rangle_\cE 
			\ = &\ \big\langle \rho (\Delta_{x_{j}}\Phi(w)\xi),P_{\cE} \rho(\Phi(v)\eta)\big\rangle_\cM 
			\ = \ \big\langle \Phi(x_{j}w)\xi,\Phi(v)\eta\big\rangle_\notV \\
			\ = & \ \langle \xi,S_{v^ *x_j w}\eta\rangle_\cH 
			\ = \ \langle \xi,S_{(x_j v)^ * w}\eta\rangle_\cH 
			\ = \ \big\langle \Phi(w)\xi,\ \Delta_{x_{j}}\Phi(v)\eta\big\rangle_{\notV} \\ 
			\ = & \  \big\langle Y_{j}^{*}\rho(\Phi(w)\xi), \rho(\Phi(v)\eta) \big\rangle_\cE.
		\end{align*}
		This proves that $Y_{j}$ are self-adjoint.
		
		\subsection{The representing vector and evaluation}\label{ssec:vec} 
		We now construct a vector $\gamma \in \cH \otimes \cE$ that realizes the functional $\varphi$.
		
		\smallskip
		
		Let $HS(\cH,\cE)$ denote the Hilbert-Schmidt operators from the Hilbert space $\cH$
		to the Hilbert space $\cE.$ For a fixed orthonormal basis $(e_n)_n$ of $\cH$, the \df{vectorization map}, $\mathrm{vec}:HS(\cH,\cE)\to \cH\otimes \cE$ is defined, for $T\in HS(\cH,\cE),$ by
		\[
		\mathrm{vec}(T)\ =\ \sum_n e_n\otimes Te_n .
		\]
		
		\smallskip
		
		We define
		\[
		\gamma\ :=\ \mathrm{vec}\big(P_{\cE} \rho\, \Phi(\varnothing)\big)\ =\ \sum\limits_{n} e_{n} \otimes P_{\cE} \rho (\Phi(\varnothing) e_{n}) \ \in\ \cH\otimes \cE, 
		\]
		where, as usual, $\varnothing$ is the empty word. Consider a word $w=x_{i_1}\cdots x_{i_k}$ with $k\le \vd+1.$ We claim that 
		\[
		Y^{w} P_{\cE} \rho\, \Phi(\varnothing) \ =\  P_{\cE} \rho\,\Phi(w).
		\]
		Indeed, for $\xi \in \cH,$ a word $|v|\le \vd$, and $1\le j\le \vg,$ 
		\[ 
		Y_j P_{\cE} \rho (\Phi(v)\xi) \ =\ Y_{j} \rho(\Phi(v) \xi) \ =\  P_{\cE} \rho(\Delta_{x_j} \Phi(v)\xi) \ =\  P_{\cE} \rho( \Phi(x_jv) \xi),
		\]
		since $\rho(\Phi(v) \xi) \in \cE.$ Hence, a finite induction argument gives,
		\begin{align*} 
			Y^{w} P_{\cE} \rho(\Phi(\varnothing) \xi) 
			\ = & \  Y_{i_{1}} \dots Y_{i_{k}} \rho(\Phi(\varnothing) \xi) \\
			\ = & \  Y_{i_{1}} \dots Y_{i_{k-1}} P_{\cE} \rho(\Delta_{x_{k}} \Phi(\varnothing) \xi)\\
			\ = & \  Y_{i_{1}} \dots Y_{i_{k-1}} \rho(\Phi(x_{k}) \xi)\\
			\ = & \  Y_{i_{1}} \rho(\Phi(x_{i_{2}} \cdots x_{i_{k}})\xi) \\
			\ = & \ P_{\cE} \rho(\Phi(w) \xi).
		\end{align*}
		Thus for any word $w \in \la x\ra_{\vd+1}$, we have
		\begin{equation} \label{eq:shift-gamma}
			(I_\cH\otimes Y^{w})\,\gamma 
			\ =\  (I_\cH\otimes Y^{w}) \Big( \sum\limits_{n} e_{n} \otimes P_{\cE} \rho(\Phi(\varnothing) e_{n})\Big) 
			\ =\  \sum\limits_{n} e_{n} \otimes P_{\cE} \rho(\Phi(w) e_{n})   
			\ =\  \mathrm{vec}\big(P_{\cE} \rho\, \Phi(w)\big).
		\end{equation}
		
		\smallskip
		
		Let $r=\sum_{w\in \II{\vd+1}}R_w\,w$ and $q=\sum_{v\in \II{\vd}}Q_v\,v$. 
		From the standard vectorization identity
		\[
		\big\langle (T\otimes I)\mathrm{vec}(A),\ (R\otimes I)\mathrm{vec}(B)\big\rangle 
		= \operatorname{Tr}\!\big(TA^{*} B R^{*}\big),
		\]
		and using \eqref{eq:Phi-kernel} and \eqref{eq:shift-gamma}, we obtain
		\begin{align*}
			\big\langle r(Y)\gamma,\ q(Y)\gamma\big\rangle 
			&= \sum_{w\in \II{\vd+1}} \sum_{v \in \la x\ra_{\vd}}
			\big\langle (R_w\otimes I)\,\mathrm{vec}(P_{\cE} \rho\, \Phi(w)),\ 
			(Q_v\otimes I)\,\mathrm{vec}(P_{\cE} \rho\, \Phi(v))\big\rangle \\
			&= \sum_{w\in \II{\vd+1}} \sum_{v \in \la x\ra_{\vd}} 
			\operatorname{Tr}\!\big(R_{w}\, (\rho\, \Phi(w))^{*} P_{\cE} \rho\, \Phi(v)\,Q_{v}^{*}\big) \\
			&= \sum_{w\in \II{\vd+1}} \sum_{v \in \la x\ra_{\vd}} \sum_{n} 
			\la P_{\cE} \rho\, \Phi(v) Q_{v}^{*}e_{n},\ \rho\,\Phi(w) R_{w}^{*} e_{n} \ra \\
			&= \sum_{w\in \II{\vd+1}} \sum_{v \in \la x\ra_{\vd}} \sum_{n} 
			\la Q_{v}^{*}e_{n},\ S_{w^{*}v} R_{w}^{*} e_{n} \ra .
		\end{align*}
		
		We now rewrite the last expression in terms of trace and use the definition of the block matrix $S$ to recover the linear functional:
		
		\begin{align*}
			\big\langle r(Y)\gamma,\ q(Y)\gamma\big\rangle
			&= \sum_{w\in \II{\vd+1}} \sum_{v \in \la x\ra_{\vd}} \sum_{n} 
			\la R_{w} S_{v^{*}w} Q_{v}^{*}e_{n},\ e_{n} \ra \\
			&= \sum_{w\in \II{\vd+1}} \sum_{v \in \la x\ra_{\vd}} 
			\operatorname{Tr}\!\big(R_{w} S_{v^{*}w} Q_{v}^{*}\big) \\
			&= \sum_{w\in \II{\vd+1}} \sum_{v \in \la x\ra_{\vd}} 
			\operatorname{Tr}\!\big(S_{v^{*}w} Q_{v}^{*} R_{w}\big).
		\end{align*}
		By \eqref{eq:block-trace}, the right-hand side equals $\varphi(q^* r)$, and hence
		\[
		\varphi(q^* r) = \langle r(Y)\gamma,\ q(Y)\gamma\rangle,
		\]
		for all $r \in \cA_{\vd+1}$ and $q \in \cA_{\vd}$.

		\subsection{Positivity of \texorpdfstring{$L(Y)$}{L(Y)}} First we show that the subspace
		\[
		\{\, q(Y)\gamma : q \in \cA _\vd \,\} \ \subseteq \ \cH \otimes \cE
		\]
		is dense in $\cH \otimes \cE$.
		
		\smallskip
		
		Fix $k \in \N$, a word $w \in \la x\ra_\vd$, and a vector $h \in \cH$. 
		Let $q \in \cA_\vd$ denote the polynomial
		\[
		q = e_k h^*\, w.
		\]
		Using \eqref{eq:shift-gamma} in the second equality, we compute:
		\begin{align*}
			q(Y)\gamma 
			&= ( e_k h^*\otimes I)\, \big( (I \otimes Y^w)\,\gamma \big) \\
			&= (e_k h^* \otimes I)\, \mathrm{vec}\!\big(P_{\cE} \rho\,\Phi(w)\big) \\
			&= \mathrm{vec}\!\big(P_{\cE} \rho\,\Phi(w)\, h e_k^* \big) \\
			&= \sum_n  e_n \otimes P_{\cE} \rho\,\Phi(w)\, h\, e_k^* e_n \\
			&= e_k \otimes P_{\cE} \rho\,\Phi(w)\, h.
		\end{align*}
		Since vectors of the form $P_{\cE}\rho\,\Phi(w)\, h$ are dense in $\cE$, and the vectors $(e_k)$ span $\cH$, it follows that
		\[
		\{\, q(Y)\gamma : q \in \cA_\vd \,\}
		\]
		is dense in $\cH \otimes \cE$.
		
		\smallskip
		
		We now prove that $L(Y) \ge 0$. By assumption, the linear functional $\varphi$ satisfies
		\[
		\varphi(q^{*} \pi(L)\, q) \ \ge\  0
		\]
		for all $q \in \cA_{\vd}$ and $\pi \in \UCP(\oS_{L},  B(\cH))$. Fix such a $\pi$. Expanding $\pi(L)$ and using linearity of $\varphi$, we obtain
		\begin{equation} \label{eq:L-positivity}
			\varphi(q^{*} \pi(L) q)
			\ =\  \varphi(q^{*}q) 
			+ \sum_{j=1}^{\vg} \varphi\big(q^{*} \pi(\bA_j) x_j q\big)
			\ \ge\  0,
			\qquad q \in \cA_\vd.
		\end{equation}
		We now apply the representation formula from Subsection~\ref{ssec:vec}. 
		For the first term,
		\[
		\varphi(q^*q) 
		\ =\  \big\langle q(Y)\gamma,\ q(Y)\gamma \big\rangle_{\cH\otimes\cE}.
		\]
		For each $j = 1, \cdots, \vg$, define $q_j := (\pi(\bA_j)x_j)q \in \cA_{\vd+1}$. Then
		\[
		\varphi(q^{*} \pi(\bA_j) x_j q)
		\ =\  \varphi(q^{*} q_j) 
		\ =\  \big\langle q_j(Y)\gamma,\ q(Y)\gamma \big\rangle_{\cH\otimes\cE} 
		\ =\  \big\langle (\pi(\bA_j) \otimes Y_j)\, q(Y)\gamma,\ q(Y)\gamma \big\rangle_{\cH\otimes\cE}.
		\]
		Substituting into \eqref{eq:L-positivity}, we obtain
		\[
		\big\langle q(Y)\gamma,\ q(Y)\gamma \big\rangle
		+ \sum_{j=1}^{\vg} 
		\big\langle (\pi(\bA_j) \otimes Y_j)\, q(Y)\gamma,\ q(Y)\gamma \big\rangle
		\ \ge\  0.
		\]
		Equivalently,
		\[ 
		\Big\langle 
		\Big(I_{\cH\otimes\cE} + \sum_{j=1}^{\vg} (\pi(\bA_j) \otimes Y_j)\Big)
		q(Y)\gamma,\ q(Y)\gamma 
		\Big\rangle_{\cH\otimes\cE} \ \geq \ 0.
		\]
		Since
		\[
		I_{\cH\otimes\cE} + \sum_{j=1}^{\vg} (\pi(\bA_j) \otimes Y_j)
		= (\pi \otimes \mathrm{id})(L(Y)),
		\]
		this shows that
		\[
		\big\langle (\pi \otimes \mathrm{id})(L(Y))\,\zeta,\ \zeta \big\rangle \ \ge\  0,
		\]
		for all $\zeta$ of the form $\zeta = q(Y)\gamma$. Since the subspace $\{\, q(Y)\gamma : q \in \cA_\vd \,\}$ is dense in $\cH \otimes \cE,$ it follows that $(\pi \otimes \mathrm{id})(L(Y)) \succeq 0$. As $\pi \in \UCP(\oS_L,  B(\cH))$ was arbitrary, $L(Y) \succeq 0$.
		\qed 
		
		\section{Proof of Theorem \ref{t:main1}} \label{sec:proofoft1}
		Since $L$ is a monic pencil, 
		$\cD_L$ has nonempty interior, i.e., it contains a neighborhood of $0$. 
		Thus the positivity of $p$ on $\cD_L$ implies that $p$ is self-adjoint. Thus,  Theorem \ref{t:main1} can be reformulated as the following proposition. 
		
		\begin{prop}
			\label{prop:SOS}
			Suppose $p \in \cA_{2\vd+1}$ is self-adjoint. If $p \notin \wtC_{\vd}$, then there exist a finite-dimensional Hilbert space $\cE_n$, a self-adjoint $\vg$-tuple $X = (X_1,\dots,X_{\vg})$ on $\cE_n$, and a vector $\gamma_n \in \cH \otimes \cE_n$ such that
			\[
			L(X) \succeq 0,
			\qquad
			\la p(X)\gamma_n,\gamma_n\ra < 0.
			\]
		\end{prop}
		
		The proof of Proposition \ref{prop:SOS} proceeds in two stages. First, we reduce to a restricted setting. We then prove the proposition under this restriction.
		
		\subsection{Reductions} \label{ssec:relax}
		We begin with a simple degree reduction.
		
		\begin{lemma}
			\label{lem:reduction}
			If $p \in \cA_{2\vd+1}$ and $p \in \wtC_{\vd+1,\vd}$, then $p \in \wtC_{\vd}$.
		\end{lemma}
		
		\begin{proof}
			Since $p \in \wtC_{\vd+1,\vd}$, we can write
			\[
			p \ =\  r^{*}r + q^{*} \pi(L) q
			\]
			for some $r \in \cA_{\vd+1}$ and $q \in \cA_{\vd}$. The term $q^{*} \pi(L) q$ has degree at most $2\vd+1$, and hence cannot cancel any degree $2\vd+2$ contribution coming from $r^{*}r$, since $p$ has degree at most $2\vd+1.$  It follows that $r \in \cA_{\vd}$ and thus
			$p\in \wtC_\vd.$
		\end{proof}
		
		We now show that if Proposition \ref{prop:SOS} holds under the additional assumption that $\cD_L$ is bounded, then it also holds when $\cD_L$ is unbounded. 
		
		\smallskip
		
		Let $S = (S_{1}, \dots, S_{\vg})$ denote the  $\vg$-tuple of $2\vg\times 2\vg$ self-adjoint matrices where the  $(2j-1,2j)$ and $(2j,2j-1)$ entries of $S_{j}$ are $1$, and all other entries of $S_{j}$ are $0.$ 
		Let $\oS$ denote the operator system spanned by
		\[
		\{I_{\tcH}\otimes I_{2\vg},\, \bA_{1} \oplus 0,\, \cdots, \bA_{\vg} \oplus 0,\, 0 \oplus S_{1},\, \dots ,\,0 \oplus S_{\vg} \}.
		\]
		For positive integers $n$ and $1\le j\le\vg,$ let
		\[
		D_{n,j} \ =\   \bA_j \oplus  \frac{1}{n} S_j  \ \in \ \oS
		\]
		and let $\penM_n$ denote the monic linear pencil,
		\[
		\penM_n \ =\  I +\sum_{j=1}^\vg D_{n,j} x_j.
		\]
		There is an $N$ sufficiently large so that $\penM_n(A)\succeq \frac12$ for all $n\ge N.$  From here on
		we consider only $n\ge N.$ The sequence $(\penM_n)$ converges coefficient-wise in the operator norm to
		the monic linear pencil
		\[
		\penM \ =\  I + \sum_{j=1}^\vg D_j x_j,
		\]
		where $D_j = \bA_j\oplus 0.$
		
		\smallskip
		
		Since $\oS_\penM$ and $\oS_{\penM_n}$ are both subsets of $\oS,$ and a cp map on either space with values
		in $ B(\cH)$ extends to a cp map on $\oS,$
		\[
		\wtC_{\vd,\penM_0} \ =\  \{\,r^*r + q^* \pi(\penM_0) q:\ r,q\in \cA_\vd,  \, \pi\in \UCP(\oS, B(\cH)) \}, 
		\]
		for $\penM_0=\penM$ or $\penM_n.$
		
		\begin{lemma} 
			\label{lem:MntoM}
			With notations as above, 
			\begin{enumerate}[\rm (1)]\itemsep=6pt
				\item \label{i:MntoM:i}
				$\cD_L=\cD_\Lambda;$ 
				\item \label{i:MntoM:ii}
				$\wtC_{\vd,L}=\wtC_{\vd,\penM};$ and 
				\item  \label{i:MntoM:iii}
				$\cap_{n \ge N} \wtC_{\vd, \penM_{n}} \ =\ \wtC_{\vd,\penM}.$
			\end{enumerate}
		\end{lemma} 
		
		\begin{proof}
			Item~\ref{i:MntoM:i} is evident from the definitions.
			To prove item~\ref{i:MntoM:ii},   first let $q\in \cA_\vd$ and $\pi\in \UCP(\oS_L, B(\cH))$ be given. Define $\wtpi:\oS\to  B(\cH)$ by 
			\[
			\wtpi(I) \ =\ I, \qquad \wtpi(\bA_j\oplus 0) \ =\   \pi(\bA_j), \qquad  \wtpi(0\oplus S_j) \ =\  0.
			\]
			
			\smallskip
			
			We now show that $\wtpi$ is completely positive. An element $Y \in \oS \otimes M_{\ell}(\mathbb{C})$ has the form
			\[
			Y \ =\  I\otimes X_0 + \sum_{j=1}^\vg (\bA_j\oplus 0)\otimes X_j + \sum_{k=1}^\vg (0\oplus S_k)\otimes X_{\vg+k} 
			\]
			for a tuple $X=(X_0,X_1,\dots,X_{2\vg})$ of $\ell \times \ell$ matrices. If $Y \succeq 0$ in $\oS \otimes M_{\ell}(\C)$ then, by item~\ref{i:MntoM:i}, it follows that
			\[
			Z \ =\  I_{\tcH} \otimes X_0 + \sum_{j=1}^\vg \bA_j \otimes X_j \ \succeq\ 0
			\]
			in $\oS_L\otimes M_\ell(\C)$. Since $\pi$ is ucp, we obtain
			\[
			(\wtpi \otimes I_\ell)(Y) \ =\  (\pi \otimes I_\ell)(Z) \ \succeq\ 0.
			\]
			Thus $\wtpi \in \UCP(\oS, B(\cH))$, and
			\[
			q^* \wtpi(\penM_\ast) q \ =\  q^* \pi(L) q,
			\]
			where $\penM_\ast$ denotes either $\penM$ or $\penM_n$. Thus $\wtC_{\vd,L}\subseteq \wtC_{\vd,\penM_\ast}$ in either case. 
			
			\smallskip

			Next let $\pi\in \UCP(\oS, B(\cH))$ and $q\in \cA_\vd$ be given. Let
			$\widetilde{\oS}$ denote the finite-dimensional operator system spanned by $\oS\cup\{I\oplus 0\}.$ 
			Since $\pi$ is ucp on $\oS$ it extends to a ucp map, still denoted $\pi,$ on $\widetilde{\oS}$. Define $\psi: \oS_L\to  B(\cH)$ by 
			$\psi(I)=\pi(I\oplus 0)$ and $\psi(\bA_j) = \pi(\bA_j\oplus 0).$ Since $\pi$ is cp, $\psi$ is cp. Indeed, if $X=(X_0,X_1,\dots,X_\vg)$
			is a tuple of $\ell \times \ell$ matrices such that
			\[
			Z = I\otimes X_0 + \sum_{j=1}^\vg \bA_j \otimes X_j \ \succeq\  0,
			\]
			in $\oS_L\otimes M_\ell(\C),$ then
			\[
			Y \ =\   (I\oplus 0) \otimes X_0 + \sum_{j=1}^\vg (\bA_j\oplus 0)\otimes X_j \succeq 0,
			\]
			in $\oS\otimes M_\ell(\C),$ and thus, 
			\[
			(\psi\otimes I_\ell)(Z) =  (\pi\otimes I_\ell)(Y) \ \succeq\ 0.
			\]
			
			\smallskip
			
			By Lemma~\ref{l:cp-to-ucp}, there is a bounded operator
			$T:\cH \to \cH$ and a ucp map $\wtpi\in \UCP(\oS_L,  B(\cH))$   such that $\psi(X)=T^* \wtpi(X)T.$ 
			Since 
			\[
			\begin{split}
				q^* \pi(\Lambda)q  & \ =\  q^*\left ( \pi(I\oplus 0) -\sum \pi(\bA_j\oplus 0) x_j  \right ) q + q^* \pi(0\oplus I) q
				\\ & \ =\  (Tq)^* \wtpi(L) (Tq) + r^*r,
			\end{split}
			\]
			where $r=  (\pi(0\oplus I))^{\frac{1}{2}} q \in \cA_{\vd},$ it follows that $q^* \pi(\Lambda)q\in \wtC_{\vd,L}.$
			Thus $\wtC_{\vd,L}=\wtC_{\vd,\penM}$, as claimed.
			
			\smallskip
			
			Note at this point it has also been demonstrated that $\wtC_{\vd,\penM}\subseteq  \wtC_{\vd,\penM_n}$ for all $n\ge N.$
			
			\smallskip
			
			To complete the proof, let $p \in \cap_{n \geq N} \wtC_{\vd, \penM_{n}}$ be given.  For each $n$ 
			there exists $r_n,q_n\in \cA_\vd$ and $\pi_n\in \UCP(\oS, B(\cH))$ such that  $p = r_n^* r_n + q_{n}^{*} {\pi}_{n} (\penM_{n}) q_{n}.$
			Trivially the sequence $(r_n^* r_n + q_{n}^{*} {\pi}_{n} (\penM_{n}) q_{n})_{n}$ converges to $p$ in the product ultraweak topology.  The sequence $(\penM_n)$ converges coefficient-wise
			in norm to $\penM.$ Hence, by Lemma~\ref{lem:pre-bounded-cone}, there exists $r,q\in \cA_\vd$ and $\pi\in \UCP(\oS, B(\cH))$ such that
			\[
			p = r^*r +q^* \pi(\penM) q \in \wtC_{\vd,\penM}.
			\]
			From \ref{i:MntoM:i} it follows that $p\in \wtC_{\vd,L}$ and the proof of item~\ref{i:MntoM:iii} and the lemma is complete.
		\end{proof} 
		
		\begin{lemma}
			\label{l:bounded-enough}
			If Proposition \ref{prop:SOS} holds in the case where $\cD_L$ is bounded, then it holds in general.
		\end{lemma}
		
		\begin{proof}
			Suppose $\cD_{L}$ is unbounded. Since $p \notin \wtC_{\vd,L},$ by Lemma~\ref{lem:MntoM},   there exists a natural number $n_{0}$ such that $ p \notin \wtC_{\vd,\penM_{n_{0}}}.$ Since $\cD_{\penM_{n_{0}}}$ is bounded, by assumption, Proposition \ref{prop:SOS} holds for $\penM_{n_{0}}.$ Thus there exists a $\vg$-tuple of self-adjoint operators $X = (X_{1}, \dots, X_{\vg})$ on some finite-dimensional Hilbert space $\cE_{n}$ and a vector $\gamma_{n} \in \cH \otimes \cE_{n}$ such that  $\penM_{n_{0}}(X)\succeq0,$ but
			\[ 
			\la p(X) \gamma_{n} , \gamma_{n} \ra_{\cH \otimes \cE_{n}} \ <\  0.
			\]
			By the definition of $\penM_{n_{0}}$ and Lemma~\ref{lem:MntoM} item~\ref{i:MntoM:i},   $\cD_{\penM_{n_{0}}} \subset \cD_{\penM}=\cD_L.$ Thus, $L(X) \succeq 0$
			and the proof is complete. 
		\end{proof}
		
		\subsection{Proof of Proposition \ref{prop:SOS}}

		In this subsection,
		we complete a proof of Proposition \ref{prop:SOS}. 
		
		\smallskip
		
		We begin with the separation argument (that does not require $\cD_L$ to be bounded).
		
		\begin{prop}\label{prop:sep}
			Let $p \in \cA_{2\vd+1}$ be such that $p \notin \wtC_{\vd}$. Then there exists a continuous (with respect to the product ultraweak topology) linear functional 
			$\varphi : \cA_{2\vd+2} \to \mathbb{C}$  such that
			\[
			\real(\varphi(p)) \ <\  0,
			\qquad
			\varphi(q) \ \ge\  0 \quad \text{for all } q \in \wtC_{\vd+1,\vd}.
			\]
		\end{prop}
		
		\begin{proof}
			By Lemma~\ref{lem:reduction}, $p \notin \wtC_{\vd+1,\vd}$. The space $\cA_{2\vd+2}$ is locally convex, and the cone $\wtC_{\vd+1,\vd}$ is closed in the product ultraweak topology (Proposition~\ref{prop:closedcone}). Hence, by the Hahn--Banach separation theorem (see \cite[Corollary 3.3.9]{D25}), there exist a continuous linear functional $\varphi$ and real numbers $\gamma_1 < \gamma_2$ such that
			\[
			\real(\varphi(p)) \ <\  \gamma_1 \ <\  \gamma_2 \ <\  \real(\varphi(q))
			\quad \text{for all } q \in \wtC_{\vd+1,\vd}.
			\]
			Since $\wtC_{\vd+1,\vd}$ is a cone of self-adjoint elements, it follows that
			\[
			\real(\varphi(p)) \ <\ 0 \ \le\  \varphi(q)
			\quad \text{for all } q \in \wtC_{\vd+1,\vd}. \qedhere
			\] 
		\end{proof}
		
		We are now ready to prove the Proposition \ref{prop:SOS}.
		
		\begin{proof}[Proof of Proposition \ref{prop:SOS}] 
			
			Taking advantage of Lemma~\ref{l:bounded-enough}, it suffices to prove
			the proposition under the additional hypothesis that $\cD_L$ is bounded.
			We divide the proof into three steps.
			
			\smallskip
			
			\noindent Step 1: Separation. 
			By Proposition~\ref{prop:sep}, there exists a product ultraweakly continuous linear functional $\varphi : \cA_{2\vd+2} \to \mathbb{C}$ such that
			\[
			\real(\varphi(p)) \ <\  0,
			\qquad
			\varphi(q) \ \ge\  0 \quad \text{for all } q \in \wtC_{\vd+1,\vd}.
			\]
			
			\smallskip
			
			\noindent Step 2: GNS construction.
			By Theorem \ref{thm:GNS} {(it is here where boundedness of $\cD_L$ is used)},
			there exist a separable Hilbert space $\cE$, a bounded self-adjoint tuple $Y = (Y_1,\dots,Y_{\vg})$ on $\cE$, and a vector $\gamma \in \cH \otimes \cE$ such that
			\[
			L(Y) \ \succeq\  0,
			\quad \text{and} \quad
			\varphi(p) \ =\  \la p(Y)\gamma,\gamma\ra \quad \text {for all $p \in \cA_{2\vd +1}$}.
			\]
			
			\smallskip
			
			\noindent
			Step 3: Finite-dimensional compression.
			Write $\gamma = \sum_{k=1}^\infty h_k \otimes f_k$, where $(f_k)$ is an orthonormal basis for $\cE$, and define
			\[
			\gamma_n \ :=\  \sum_{k=1}^n h_k \otimes f_k.
			\]
			Since $\real(\varphi(p)) < 0$, there exists $n$ such that
			\[
			\la p(Y)\gamma_n,\gamma_n\ra < 0.
			\]
			Let $\cE_n := \mathrm{span}\{f_1,\dots,f_n\}$ and let $P_{\cE_n}$ denote the orthogonal projection onto $\cE_n$. Define
			\[
			X_j \ :=\  P_{\cE_n} Y_j|_{\cE_n}, \quad j=1,\dots,\vg.
			\]
			Then
			\[
			L(X) \ =\  (I_{\cH} \otimes P_{\cE_n})\, L(Y)\,|_{\cH \otimes \cE_n} \ \succeq\  0,
			\]
			while
			\[
			\la p(X)\gamma_n,\gamma_n\ra
			\ =\  \la p(Y)\gamma_n,\gamma_n\ra \ <\  0. \qedhere
			\] 
		\end{proof}
		
		As a consequence of Theorem~\ref{t:main1}, we obtain a closedness property of the cone. {We note for clarity that the corollary below does not require \(\cD_L\) to be bounded.}

		\begin{corollary}
			\label{cor:WOT closed}
			The convex cone $\wtC_{\vd,L}$ is closed in the product WOT.
		\end{corollary}
		
		\begin{proof}
			Let $(p_\alpha)$ be a net in $\cA_{2\vd+1}$ that converges to $p$ in the product WOT. For each $X \in \cD_L$, the operators $p_\alpha(X)$ are positive semidefinite and converge in WOT to $p(X)$. Since WOT convergence preserves positivity, $p(X) \succeq 0$. The result now follows from Theorem~\ref{t:main1}.
		\end{proof}

		\section{The finite-dimensional setting}\label{sec:fd}
		This section presents the sums of squares representations obtained from  Theorem~\ref{t:main1} in the cases
		that either one (or both) of $\cH$ and $\tcH$ are finite dimensional. See Theorem~\ref{c:HKM12}.
		When both are finite dimensional the main result of \cite{HKM12} is recovered;
		and when $\cH$ is finite dimensional and $\tcH,$ the space that the coefficients
		of $L$ act on, is finite dimensional,  \cite[Theorem 1.5]{Zal17} is obtained.  
		Here we add the bound $\nu^2 N(d)$ implicit there.
		
		\smallskip
		
		Theorem~\ref{c:HKM12}  uses the following conventions. 
		In the case that  $\tcH$ is finite dimensional  we let   $L=\oplus_{k=1}^K L_k$ 
		denote a direct sum decomposition of $L.$  It is not assumed that $L$ so written
		is fully reduced and thus $K$ can always be taken to be $1.$  The corresponding Hilbert space decomposition 
		is written as $\tcH =\oplus_{k=1}^K \tcH_k$ and the dimensions of the $\tcH_k$
		are denoted by $\mu_k.$ Hence $\mu=\sum_{k=1}^K \mu_k$  is the dimension of $\tcH.$

		\begin{thm}%
			\label{c:HKM12}
			Let $\cF$ denote a separable Hilbert space. In the case $\cF$ is finite dimensional, let $\nu=\dim \cF.$
			Suppose  %
			$p\in   B({\cF}) \otimes \C\la x \ra_{2\vd+1}$  and $p(X) \succeq 0$ for every $X\in \cD_L.$ 
			
			\begin{enumerate}[\rm (i)]\itemsep=5pt
				\item \label{i:cHKM:i}
				If $\cF$ is finite dimensional and $\tcH$ is infinite dimensional, then there is a Hilbert space $\cE$ 
				of dimension at most $\nu^3  N(\vd),$  a ucp map $\pi:\oS_L \to  B(\cE),$ 
				and polynomials $r, q \in  B(\cF,\cE)\otimes \C\la x \ra_{\vd}$ 
				such that $p = r^* r + q^* \pi(L) q.$

				\item \label{i:cHKM:iii}
				If both $\cF$ and $\tcH$ are finite dimensional,  
				then there exist Hilbert spaces $\cE_k$ of dimension at most $\nu \mu_k N(\vd)$
				and   $r_{k,j}, q_{k,j} \in  B(\cF,\cE_k\otimes \tcH_k)\otimes \C\la x \ra_{\vd}$ for $1\le k \le K$ and $1\le j \le \nu \mu_k N(\vd)$ 
				such that
				\[
				p =\sum_{k=1}^K \sum_{j=1}^{N_k}  r_{k,j}^* r_{k,j} + \sum_{k=1}^K\sum_{j=1}^{N_k} q_{k,j}^* L_k  q_{k,j},
				\]
				where $N_k\le \nu \mu_k^2 N(\vd).$ 
				In particular, if each $L_k$ is scalar-valued, then there are at most  $\nu \mu  N(\vd)$ 
				many  polynomials  $r_{k,j}$ and at most $\nu \mu N(\vd)$ many  $q_{k,j}$ 
				and these polynomials 
				can be identified with elements of  $M_{1,\nu}(\C) \otimes \C\la x \ra_{\vd}.$

				\item  \label{i:cHKM:ii}
				If $\cF$ is infinite dimensional and $\tcH$ is finite dimensional, then, 
				for a Hilbert space $\cE$ such that $\cF = \cE\otimes \tcH,$ 
				there exist polynomials $r_k, q_k \in   B(\cF,\cE\otimes \tcH_k)\otimes \C\la x \ra_{\vd} $ 
				such that 
				\[
				p  = \sum r_k^* r_k +  \sum q_k^* (I_\cE \otimes L_k) q_k.
				\]
			\end{enumerate}
		\end{thm}  
		
		\begin{proof}  
			To prove item~\ref{i:cHKM:i}, 
			let $\cH = \ell^2\otimes \cF,$  
			where $\ell^2=\ell^2(\N)$ is the usual space of $\ell^2$ sequences  $a=(a_m)_{m=0}^\infty.$ 
			Thus $\cH$ is separable and infinite dimensional.   Let $\iota: \cF\to \cH$ 
			denote the isometry $\iota f  = \zeta \otimes f$  for $f\in\cF,$  where $\zeta= (\zeta_{n})_{n \geq 0} \in \ell^2$ is the sequence with $\zeta_0=1$ and $\zeta_j=0$ for $j>0.$ 
			Let $\hat{p} = \iota  p \iota^*\in  B(\cH)\otimes \C\la x\ra_{2d+1} =\cA_{2\vd+1}.$   By construction, $\hat{p}(X)\succeq 0$
			for all $X\in \cD_L.$ Thus, by {Theorem~\ref{t:main1}}, there exists $\hat{r}, \, \hat{q}\in \cA_\vd$ 
			and $\pi\in \UCP(\oS_L, B(\cH))$ such that  
			\[ 
			\hat{p} \ =\  \hat{r}^* \hat{r} + \hat{q}^*  \pi(L) \hat{q}.
			\]
			It follows that,
			\begin{equation}
				\label{eq:HKM:2}
				p \ =\  r^* r + q^* \pi(L) q,
			\end{equation}
			where $r= \hat{r} \iota$ and $q= \hat{q}\iota.$  In particular, the coefficients of $r$ and $q$ map $\cF$ into $\cH.$
			The span of the ranges of the coefficients of $q$ has dimension  at most $\nu N(\vd).$  Since these ranges lie in $\ell^2\otimes \cF$
			and $\cF$ has dimension $\nu,$ 
			there is a subspace $\cS(q)$ of $\ell^2$ of dimension at most $\nu^2  N(\vd)$ such that the ranges of the coefficients of $q$
			lie in $\cS(q)\otimes \cF.$   Similarly, there is a subspace $\cS(r)$ of $\ell^2$ of dimension at most $\nu^2 N(\vd)$ 
			such that the ranges of the coefficients of $r$ lie in $\cS(r)\otimes \cF.$ By enlarging $\cS(q)$ and replacing
			$r$ with $Ur$ for an appropriate unitary as needed, we may (and do)
			assume $\cS(r)$ is $\cS(q).$  Let $\cE=\cS(q)\otimes \cF$ and let $V:\cE\to \ell^2\otimes \cF =\cH$ denote the inclusion.  Thus
			$V^* \pi(L) V: \oS_L\to  B(\cE)$ is ucp and 
			\[ 
			p \ = \ r^* r + q^* V^* \pi(L) V q, 
			\]
			where, without loss of generality, $r,q\in  B(\cF, \cE)\otimes \C\la x \ra_{\vd} .$  Finally the dimension of $\cE$ is at most
			$\dim \cF\, \dim S(q) \le \nu^3 N(\vd).$
			
			\smallskip
			
			Turning to  item~\ref{i:cHKM:iii}, since every unital $\ast$-representation 
			of $M_\mell(\C)$
			on Hilbert space is a multiple of the identity representation,  
			if $\pi \in \UCP(\oS_L, B(\cH)),$  then there is an auxiliary Hilbert  $\cE,$  an isometry $V:\cH\to \cE\otimes\C^\mell$ 
			and unital $\ast$-representation $\psi:M_\mell(\C)\to  B(\cE\otimes \C^\mell)$ such that $\psi(S) = I_\cE\otimes S$ and
			\begin{equation}
				\label{eq:HKM:1}
				\pi(S)  \ =\  V^* \psi(S) V \ =\  V^* (I_\cE \otimes S) V,
			\end{equation}
			for all $S\in \oS_L.$  Thus, since the proof of item~\ref{i:cHKM:i} was agnostic about
			whether $\tcH$ is finite or infinite dimensional, in this case still with $\cH=\ell^2\otimes \cF,$  equation~\eqref{eq:HKM:2} becomes,
			\begin{equation}
				\label{eq:HKM:2.5}
				p \ =\  r^* r + q^* (I_{\cE} \otimes L) q.
			\end{equation}
			
			Let $r_w$ and $q_w$ denote the coefficients of $r$ and $q$ respectively.  The range of each coefficient is a subspace of $\ell^2\otimes \cF$
			of dimension  at most $\nu$  and there are at most $N(\vd)$ of each. For $1\le k\le K$ choose an
			orthonormal  basis $\{e_{k,j} \mid 1\le k\le K, \, 1\le j\le \mu_k\}$ of $\tcH_k$  where, for fixed $k,$
			$\{e_{k,1},\dots, e_{k,\mu_k} \}$ is an orthonormal basis  
			of $\tcH_k.$ Let 
			\[
			\cS_k(s) = \spann \{ \gamma_{k,j} \in \ell^2 \mid \gamma \in \bigcup_w \range s_w, \,  (I\otimes P_k)  \gamma =\sum_{j=1}^{\mu_k} \gamma_{k,j} \otimes e_{k,j}\}  \subseteq \ell^2 
			\]
			for $s=r,q,$ where $P_k =\sum_{j=1}^{\mu_k} e_{k,j} e_{k,j}^*$ is the projection onto $\tcH_k.$
			In particular, the dimension of $\cS_k(s)$ is at most $\nu \mu_k N(\vd)$ and the range of each $s_w$
			lies in $\oplus_k (\cS_k(s) \otimes  \tcH_k).$  By enlarging $\cS_k(r)$  and 
			by 
			replacing $r$ with
			$Ur$ for an appropriate choice of unitary $U$  as needed,
			it may be (and is) assumed that $\cS_k(r)$ is  $\cS_k(q).$   Let $\cE_k=\cS_k(q).$
			
			\smallskip
			
			Let $\{u_{k,1},\dots,u_{k,\mu_k}\}$ denote an orthonormal basis for $\cE_k$  and set  $s_{k,j} = (u_{k,j}^* \otimes I_{\tcH_k}) s$,
			where $s$ is either $r$ or $q$, and $I_{\tcH_k}$ is the identity on $\tcH_k$.  With $I$ the identity of $\tcH,$ 
			\begin{equation}
				\label{eq:HKM:3}
				r^*  (I_\cE \otimes I_{\tcH})  r\ =\  r^*  \big ( \bigoplus_k (I_{\cE_k}\otimes I_{\tcH_k})\big )  r \ = \ 
				r^* \Big (\bigoplus_k  \big((\sum_{j=1}^{\mu_k} u_{k,j} u_{k,j}^*)\otimes I_{\tcH_k}\big)\Big ) r
				= \sum_{k=1}^K  \sum_{j=1}^{N_k}  r_{k,j}^*  r_{k,j},
			\end{equation}
			and similarly, 
			\begin{equation}
				\label{eq:HKM:4}
				q^* (I_\cE \otimes L) q \ = \   q^*  \big ( \bigoplus_k ( \sum u_{k,j} u_{k,j}^*) \otimes L \big ) q 
				=  \sum_{k=1}^K  \sum_{j=1}^{N_k} q_{k,j}^* L_k q_{k,j}.
			\end{equation}
			Combining equations~\eqref{eq:HKM:2.5}, \eqref{eq:HKM:3}  and \eqref{eq:HKM:4}
			completes the proof of item~\ref{i:cHKM:iii}. 
			
			\smallskip
			
			To prove item~\ref{i:cHKM:ii} modify the proof of item~\ref{i:cHKM:iii} as follows.
			Let $P_k$ denote the projection of $\tcH$ onto $\tcH_k,$  choose $r_k= (I_\cE \otimes P_k) r$
			and $q_k = (I_\cE \otimes P_k) q$ so that, for instance, $q=\oplus_k q_k,$ and substitute into 
			equation~\eqref{eq:HKM:2.5} using $P_k L P_k =L_k.$
		\end{proof} 
		
		\begin{remark}\rm
			\label{r:HKM12}
			In item~\ref{i:cHKM:iii}  let $\cE=\oplus \cE_k$. It is a simple matter to construct  
			$r_\ell \in  B(\cE\otimes \tcH)\otimes \C\la x\ra_{\vd}$    for $1\le \ell\le \nu$ such that
			$\sum_\ell r_\ell^* r_\ell = \sum_{k,j} r_{k,j}^* r_{k,j}.$ 
			
			Similarly in item~\ref{i:cHKM:ii} 
			there is an $r\in  B(\cF,\cE\otimes \tcH)\otimes \C\la x \ra_{\vd} $ such that
			$r^*r =\sum_k r_k^* r_k.$
		\end{remark}

		\section{Not necessarily monic pencils and affine linear change of variable}
		\label{sec:aff}
		The condition that $L$ is a monic linear pencil can be relaxed in several different ways. 
		Here we consider an affine change of variable  tailored for use in applying  Proposition~\ref{prop:SOS}
		in the proof of Theorem~\ref{t:posforpovmm} below,  a result used in the proof of  Theorem~\ref{t:main2}.
		
		\smallskip
		
		Let $\nmL$ be a given linear pencil (not necessarily monic) with self-adjoint coefficients in $ B(\tcH),$
		\[
		\nmL(x) \ =\  \bA_0 +\sum_{j=1}^\vg \bA_j x_j .
		\]
		Suppose $T$ is an invertible $\vg\times\vg$ real matrix and $b\in \mathbb{R}^\vg.$ The pair $(T,b)$ gives rise to the change of variables,
		\[
		x \ \mapsto\  y \ =\ Tx+b,
		\]
		where $y=(y_1,\dots,y_\vg)$ and $y_j = \sum_{k} T_{j,k} x_k + b_j.$ Let 
		\[
		\affL(x) \ =\  \nmL(Tx+b) \ =\ \Bigl(\bA_{0}+\sum_{i=1}^{\vg} b_{i}\bA_{i} \Bigr) + \sum_{j=1}^{\vg} \Bigl( \sum_{i=1}^{\vg} T_{i,j} \bA_{i} \Bigr) x_{j} .
		\]
		Thus $\affL$ is a linear pencil.  Note that if $\affL$ is monic, then the identity
		is in the vector space $\oS$ spanned by  $\{\bA_0,\bA_1,\dots,\bA_\vg\}$
		and $\oS =\oS_{\affL}$ is an operator system.
		
		\begin{proposition} 
			\label{prop:aff}
			Suppose $T$ is an invertible $\vg\times \vg$ matrix,  $b\in \R^\vg$ and  
			\[
			\affL(x) \ =\  \nmL(Tx+b) \ =\  I +\sum \affbA_j x_j
			\]
			is a monic linear pencil, $\cF$ is a separable Hilbert space
			and  $p\in  B(\cF)\otimes \cA_{2\vd+1}.$
			
			If both $\cF$ and $\tcH$ are  infinite dimensional
			and  $p(X)\succeq0$ for $X\in \cD_\nmL,$ then 
			there exists
			$r,q\in B(\cF)\otimes \cA_\vd$ and a ucp map $\pi:\oS_{\affL} \to B(\cF)$ such that
			\[
			p \ =\  r^*r + q^* \pi(\nmL)q.
			\]
			
			If either $\cF$ or $\tcH$ are finite dimensional and $p(X)\succeq0$ for $X\in \cD_\nmL,$ then 
			the conclusions of  Theorem~\ref{c:HKM12} hold for $p$ and $\nmL.$
		\end{proposition}
		
		\begin{proof}
			Let $\affp(x) = p(Tx+b).$  Thus $\affp \in \cA_{2\vd+1}$ and $\affp(X)\succeq0$
			for $X\in\cD_{\affL}.$   If both $\cF$ and $\tcH$ are infinite dimensional, then 
			by Theorem~\ref{t:main1},  there exists $\affr,\affq\in \cA_\vd$ and 
			a $\pi \in \UCP(\oS_{\affL}, B(\cH))$ such that 
			\[
			\affp \ =\  \affr^* \affr + \affq^* \pi(\affL) \affq.
			\]
			Setting $r=\affr(T^{-1}(x-b))$ and $q=\affq(T^{-1}(x-b))$ gives,
			\[
			p \ =\  r^* r + q^* \pi(L) q.
			\]
			
			If either $\cF$ or $\tcH$ is finite dimensional  then the same change of variable gives the desired conclusion
			after noting that the change of variable commutes with any choice of direct sum decomposition
			of $\nmL.$
		\end{proof}

		\subsection{A special linear pencil and an application of Proposition \ref{prop:aff}} 
		
		This subsection presents the application of Proposition~\ref{prop:aff} used in the proof of Theorem~\ref{t:posforpovmm}, which is subsequently employed in the proof of Theorem~\ref{t:main2}.
		
		\smallskip
		
		Let
		$\fqy=(\fqy_1,\dots,\fqy_{n-1})$ denote an $n-1$ tuple of freely non-commuting self-adjoint variables.
		Let \df{$\snmL[n]$} denote the  linear matrix polynomial (linear pencil) \index{$L[n]$}
		\begin{equation}\label{eq:L_ndef}
			\snmL [n](\fqy)
			\ =\  \begin{pmatrix}
				\fqy_1 & & & \\
				& \ddots & & \\
				& & \fqy_{n-1} & \\
				& & & 1 - \sum_{i=1}^{n-1} \fqy_{i}
			\end{pmatrix}.
		\end{equation}
		For a Hilbert space $\cE,$ we define \index{$\cD_n(\HE)$}
		\[
		\cD_{n}(\HE) \ =\ 
		\Bigl\{ E=(E_1,\dots,E_{n-1}):\ E_{i} \in  B(\cE)_{\rm sa}, \  E_{i} \succeq 0, \ 1 \succeq \sum_{i=1}^{n-1} E_i \Bigr\},
		\]
		where \df{$ B(\cE)_{\rm sa}$}  denotes the (bounded) self-adjoint operators on $\HE.$
		The \df{free spectrahedron} associated to the linear pencil $\snmL [n]$ is the sequence
		$\cD_n= (\cD_n(\C^\ell))_\ell.$ %
		Observe that if $E\in \cD_n(\HE),$ then, setting $E_n=I-\sum_{i=1}^{n-1} E_i,$  each $E_i$ is psd and
		\[
		I \ =\ \sum_{i=1}^n E_i
		\]
		so that $(E_1,\dots,E_n)$ corresponds to a  positive $\BO(\HE)$-valued  measure (\df{$povm$})
		on the set $\{1,2,\ldots,n\}.$
		
		\smallskip
		
		Fix  a positive integer $m$ and positive integers $n_1,\dots,n_m\ge 2.$ Let 
		\[
		y \ =\ (y_{1,1},\dots,y_{1,n_1-1}, y_{2,1},\dots, y_{m,n_m-1})
		\]denote freely non-commuting self-adjoint variables. 
		For notational convenience, let 
		\[
		\snmL_i=\snmL [n_i], \quad \snmL =\oplus_{i=1}^m  n_i \snmL_i.
		\]
		For a Hilbert space $\cE,$ we have 
		\begin{equation}
			\label{eq:DLE}
			\cD_{\snmL}(\cE)
			=
			\left\{
			(E_{i,j}) :
			\begin{array}{l}
				
				1\le i\le m,\quad  1\le j\le n_i-1, \\[4pt]
				E_{i,j}\in  B(\cE)_{\rm sa}, \quad  E_{i,j}\succeq 0, \quad  I \succeq \displaystyle\sum_{j=1}^{n_i-1} E_{i,j}
			\end{array}
			\right\}.
		\end{equation}
		The free spectrahedron associated to the linear pencil $\snmL$ takes the form $\cD_\snmL = (\cD_\snmL(\C^\ell))_\ell.$ 
		
		\smallskip 
		
		To see that there is an affine linear transformation that converts $\snmL$ to a monic linear pencil $\affL,$ 
		let $I_i$ denote the $(n_i-1)\times (n_i-1)$ identity matrix and  let $b_i$ denote the vector in $\mathbb{R}^{n_{i}-1}$ with entries $\frac{1}{n_i}$.  Let $T_{i} = I_{i},$ $T = \oplus_{i=1}^m I_i,$ and $b=\oplus_{i=1}^m b_i.$ Thus, with $y_{i,*} = (y_{i,1},\dots,y_{i,n_i-1}),$ \index{$y_{i,*}$}
		\[
		n_i \snmL_i(T_i y_{i,*} +b_i) \ =\   I_i  + 
		n_i \, \begin{pmatrix}
			y_{i,1} & & & \\
			& \ddots & & \\
			& & y_{i,n_i-1} & \\
			& & &  - \sum_{j=1}^{n_i-1} y_{i,j}
		\end{pmatrix},
		\]
		and  $\affL = \oplus_{i=1}^{m} n_{i} \snmL_{i}(T_{i} y_{i,*} + b_{i})$ is monic. 
		
		\smallskip
		
		The following result interprets Proposition~\ref{prop:aff}  concretely for the pencil $\snmL$ above
		taking $\vg=\sum_{i=1}^m  (n_i-1).$
		
		\begin{proposition}
			\label{prop:WS}
			Let $\cF$ denote a separable Hilbert space.
			If $p\in B(\cF)\otimes \C \la y\ra_{2\vd+1}$
			is  self-adjoint, then $p(E)\succeq 0$ for all $E\in \cD_\snmL $ if and only if the following hold:
			
			\begin{enumerate}[\rm (i)]\itemsep=5pt
				\item 
				\label{i:WS:i}
				If $\cF$ is infinite  dimensional, then there exist
				$f\in B(\cF)\otimes \C \la y\ra_{\vd} $
				and $f_i, f_{i,j}\in  B(\cF,\C)\otimes \C \la y\ra_\vd$ such that 
				\[
				p \ =\  f^*f + \sum_{i=1}^m \left[ \sum_{j=1}^{n_i-1} f_{i,j}^* y_{i,j}f_{i,j}
				+ f_i^*\Bigl(1-\sum_{j=1}^{n_i -1}y_{i,j}\Bigr)f_i \right].
				\]
				
				\item 
				\label{i:WS:ii}
				If $\cF=\C^\nu$ is finite-dimensional,   then there exist 
				\[
				g_{k,i,j},f_{i,k},f_{i,j,k}\in M_{1,\nu}(\C)\otimes \C \la y \ra_{\vd}
				\]
				such that
				\[
				p \ =\  \sum_{k=1}^\mu \sum_{i=1}^m \sum_{j=1}^{n_i}  g_{k,i,j}^* g_{k,i,j}
				+ \sum_{k=1}^\mu \sum_{i=1}^m
				\left[\sum_{j=1}^{n_i-1} f_{i,j,k}^* y_{i,j} f_{i,j,k}
				+ f_{i,k}^*\Bigl(1-\sum_{j=1}^{n_i-1} y_{i,j}\Bigr)f_{i,k}\right].
				\]
			\end{enumerate}
		\end{proposition}
		
		\begin{proof}
			The backward implication is immediate. 
			
			\smallskip
			
			To prove the forward implication, note that the pencil $\snmL$ has the direct sum decomposition
			$\bigoplus_{i=1}^m \Big((\bigoplus_{j=1}^{n_i-1} \snmL_{i,j}) \oplus \snmL_{i,n_i} \Big)$
			for the scalar pencils $\snmL_{i,j}(y) = n_i y_{i,j}$ for $1\le i< n_i$ and $\snmL_{i,n_i}(y) = 1- n_i \sum_{j=1}^{n_i} {y_{i,j}}.$ 
			By  Proposition~\ref{prop:aff}, the conclusions of Theorem~\ref{c:HKM12} hold
			giving the conclusion of item~\ref{i:WS:i} or  \ref{i:WS:ii}  depending
			on whether $\cF$ is  infinite or finite dimensional, respectively.
		\end{proof}

		\section{Positivstellensatz for the \texorpdfstring{$*$-}aalgebra \texorpdfstring{$\povm(n)$}{POVMn}} \label{sec:povm} 
		
		\subsection{The \texorpdfstring{$*$-}aalgebra povm} 
		Given a positive integer $n\ge 2,$ this subsection describes the construction of a $*$-algebra $\povm(n)$
		naturally associated to positive operator-valued measures on the set $\{1,2,\dots,n\}.$   
		
		\smallskip
		
		Let
		$y=(y_1,\dots,y_{n-1})$ denote an $n-1$ tuple of freely non-commuting self-adjoint variables and let $\C\la y\ra$ denote the resulting unital free $*$-algebra. 
		Let $L[n]$ denote the  linear matrix polynomial (linear pencil) defined in \eqref{eq:L_ndef}.
		
		\smallskip
		
		Let 
		\[
		\HG_n \ =\  \oplus_{\ell=1}^\infty \oplus_{E\in \cD_n(\C^\ell)} \C^{\ell}
		\]
		and define $\Psi_n:\C\la y\ra\to  B(\HG_n)$ by \index{$\Psi_n$} \index{$ B(\HG_n)$}
		\[
		\Psi_n(p) \ =\  \oplus_{E\in\cD_n} \, p(E).
		\]
		Note that \(\Psi_n\) is a \(*\)-homomorphism. Since \(\cD_n(\C^\ell)\) has nonempty interior for every \(\ell\), the map \(\Psi_n\) is faithful. 
		In particular, $\povm(n)$ can be viewed as a $*$-subalgebra of $ B(\HG_n).$ Consequently,
		\[
		\|p\|\ :=\ \|\Psi_n(p)\| \qquad (p\in \C\la y\ra)
		\]
		defines a norm on $\C\la y\ra$. Equipped with this norm, $\C\la y\ra$ becomes a pre-\(C^*\)-algebra, which we denote by \df{$\povm(n)$}.
		Its completion (a $C^{*}$-algebra) is denoted by 
		\df{$\POVMn$}. For a more detailed discussion, we refer the reader to \cite{Cim09, Oz13}.
		
		\smallskip
		
		Each $E\in \cD_n(\ell)$ induces a unital $*$-representation 
		\[
		\rep_E:\povm(n)\to M_\ell(\C), \qquad p \ \mapsto\ p(E)
		\]
		for $p\in \C \la y \ra.$ Since $\rep_E$ is a bounded $*$-homomorphism on the dense subalgebra $\C \la y \ra$ of $\POVMn,$ it extends to a $*$-representation of $\POVMn.$
		
		\begin{lemma} \label{l:repCpovm}
			If $\rep:\POVMn\to  B(\HE)$ is a $*$-representation of $\POVMn$ on a Hilbert space $\HE,$ then there is a tuple $Y\in \cD_n(\HE)$ such that $\rep=\rep_Y.$ 
		\end{lemma}
		
		\begin{proof}
			The set $\povm(n)$ is dense in $\POVMn$ and is generated by the tuple $y.$ Thus, setting $Y_j=\rep(y_j),$ we have 
			\[
			\rep(p) \ =\ p(Y)
			\]
			for every $p \in \C\la y\ra.$ It remains to see that $Y \in \cD_{n}(\HE).$ 
			
			\smallskip
			
			Let $f_j(y)=y_j.$ Note that $\rep_E(f_j)$ is psd for each $E\in\cD_n.$ Thus so is $\Psi_{n}(f_j),$ which means $f_j$ is positive (since $\Psi_{n}$ is faithful)  
			as an element of $\POVMn.$ Thus  $\rep(f_j)=Y_j$ is psd. A similar argument applied to  $f(y)=\sum_{j=1}^{n-1} y_j$ shows $\Psi_n(f)$ is psd and contractive and therefore so is $\rep(f)=\sum Y_j.$ Hence $Y\in \cD_{n}(\HE).$
		\end{proof}

		\begin{lemma}
			\label{l:infvfin}
			If $\cE$ is a Hilbert space, and $Y=(Y_1,\dots,Y_{n-1})$ is a tuple of psd operators on $\cE$ that satisfies 
			the inequality,
			\[
			I \ \succeq\  \sum_{j=1}^{n-1} Y_j, 
			\]
			then, for each positive integer $d$ and unit vector $e\in\cE,$
			there exists an $\ell,$ an  $E\in\cD_n(\C^\ell)$ and a unit vector
			$\xi\in \C^\ell$  such that 
			\[
			\|p(Y)e\| \ =\ \|p(E)\xi\|
			\]
			for all $p\in \C \la y\ra_d.$ In particular, $\|p(Y)\|\le \|p\|$ for all $p\in \C \la y\ra.$
		\end{lemma}  
		
		\begin{proof}
			The set  $\cF=\{p(Y)e:\deg(p)\le d\}$
			is a finite-dimensional subspace of $\cE.$  Denote
			$\ell=\dim\cF$. Since $1(Y)e=e$, we have $e\in\cF$.
			
			\smallskip
			
			Let $V:\cF\hookrightarrow \cE$ denote the inclusion map and define
			\[
			E_j \ =\  V^*Y_jV, \qquad j=1,\dots,n-1,
			\]
			and write $E=(E_1,\dots,E_{n-1})$. Since $Y_j\succeq0$ and
			$\sum_{j=1}^{n-1}Y_j\preceq I$, the same inequalities hold after
			compression. Hence $E\in\cD_n(\C^\ell)$. Set $\xi=e$, which is a unit
			vector in $\cF\cong\C^\ell$.
			
			\smallskip
			
			We now show that $p(E)\xi=p(Y)e$ for every $p\in\C\langle y\rangle_d$.
			By linearity it suffices to consider a word $w=y_{i_1}\cdots y_{i_r}$
			with $r\le d,$ in which case,
			\[
			w(E)e \ =\ (V^*Y_{i_1}V)\cdots(V^*Y_{i_r}V)e.
			\]
			For each $k=0,\dots,r-1$, the vector
			\[
			Y_{i_{k+1}}\cdots Y_{i_r}e
			\]
			lies in $\cF$, since it is of the form $q(Y)e$ for a word $q$ of degree
			at most $d$. Hence $VV^*$ acts as the identity on these vectors, and
			therefore
			\[
			w(E)e
			\ =\ V^*Y_{i_1}(VV^*)Y_{i_2}\cdots(VV^*)Y_{i_r}e
			\ =\ V^*Y_{i_1}\cdots Y_{i_r}e.
			\]
			Since $w(Y)e\in\cF$, we also have
			\[
			V^*Y_{i_1}\cdots Y_{i_r}e \ =\ w(Y)e.
			\]
			Thus $w(E)e=w(Y)e$, and by linearity
			\[
			p(E)e \ =\  p(Y)e
			\]
			for all $p\in\C\langle y\rangle_d$. Consequently,
			\[
			\|p(Y)e\| \ = \|p(E)\xi\|,
			\]
			which proves the lemma.
		\end{proof}

		\subsection{Free products}
		Fix a positive integer \(m\) and integers $ n_1,\dots,n_m \ge 2.$ Let
		\[
		y \ =\  \{\, y_{i,j}:\quad  1\le i\le m,\; 1\le j\le n_i-1\}
		\]
		be freely noncommuting self-adjoint variables, and let
		\(\C\langle y\rangle\) denote the corresponding free $*$-algebra of
		polynomials in $y.$
		
		\smallskip
		
		Let $\underline{n} = (n_{1}, \dots,n_{m}).$ As in Section~\ref{sec:aff}, let $L_i=L[n_i]$ and $L=\oplus n_i L_i.$ Following the construction of $\povm(n),$ let 
		\[
		\HG \ =\ \oplus_{\ell=1}^\infty  \oplus_{E\in\cD_L(\C^\ell)} \C^\ell,
		\]
		define $\Psi:\C \la y \ra \to  B(\HG)$ by \index{$\Psi$} \index{$ B(\HG)$}
		\begin{equation} \label{eq:PsiHG}
			\Psi(p) \ =\  \oplus_{E \in \cD_{L}} \, p(E),
		\end{equation}
		and denote the resulting pre-$C^{*}$-algebra by $\povm(\underline{n})$ and its completion by $\POVM(\underline{n}).$ In this way, $\POVM(\underline{n})$ is naturally a sub-$C^{*}$-algebra of $ B(\HG).$
		
		\smallskip
		
		Setting, for each \(i\) 
		\[
		y_i \ = (y_{i,1},\dots,y_{i,n_i-1}),
		\]
		there is a canonical identification
		\begin{equation}
			\label{eq:Clx-ast}
			\C\langle y\rangle
			\;=\;
			\C\langle y_1\rangle * \cdots * \C\langle y_m\rangle
		\end{equation}
		as unital $*$-algebras  induced by the map defined on alternating products by
		\[
		p_{i_1}(y_{i_1}) * \cdots * p_{i_k}(y_{i_k}) 
		\ \mapsto \ 
		p_{i_1}(y_{i_1})\cdots p_{i_k}(y_{i_k}),
		\qquad
		p_{i_j}\in \C\langle y_{i_j}\rangle .
		\]
		
		\smallskip
		
		In what follows, $ \cbast $ denotes the universal free product of $C^*$-algebras. More precisely, if $ \cB_1 $ and $ \cB_2 $ are unital $ C^*$-algebras, then $\cB_1 \cbast \cB_2$ denotes the unital $C^*$-algebra obtained by completing the algebraic free product $ \cB_1 \ast \cB_2$ with respect to the universal norm; see \eqref{eq:univnorm}. We refer the reader to \cite{VDN92} for a detailed discussion on the free product of $C^{*}$-algebras. 
		
		\begin{proposition} \label{p:twoways}
			The canonical identification of the  $*$-algebras in equation~\eqref{eq:Clx-ast} induces the identifications
			\begin{enumerate}[\rm (1)]\itemsep=5pt 
				\item  $\povm(\underline{n}) \ =\    \povm(n_1) \ast \cdots \ast \povm(n_{m});$  %
				\item $\Psi \ =\  \Psi_{n_1} \ast \cdots \ast \Psi_{n_{m}}: \C\la y\ra \to \povm(\underline{n});$ and 
				\item  $\POVM(\underline{n}) \ =\  \POVM(n_{1}) \cbast \cdots \cbast \POVM(n_{m}).$ 
			\end{enumerate}
		\end{proposition}
		
		The proof of Proposition~\ref{p:twoways} will use the following analog of Lemma~\ref{l:infvfin}.
		
		\begin{lemma}
			\label{l:infvfin+}
			If  $\HE$ is a Hilbert space,  $Y\in \cD_L(\HE)$ and $p\in \C \la y\ra,$ then 
			\[
			\|p(Y)\| \ \le\  \|p\|.
			\]
		\end{lemma}

		\begin{proof}[Proof of Proposition~\ref{p:twoways}]
			Items (1) and (2) follow immediately from the definitions and the canonical identification of the free algebras in \eqref{eq:Clx-ast}.
			
			To prove (3), we must show that the norm on $\povm(\underline{n})$, inherited from its embedding in $ B(\HG)$ via $\Psi$, coincides with the universal free product $C^*$-norm on the algebraic free product $\ast_{i=1}^m \povm(n_i)$.
			By definition, the norm of an element $p \in \ast_{i=1}^{m}\povm(n_i)$ in the free product $C^*$-algebra $\POVM(n_1)\cbast\cdots\cbast \POVM(n_m)$ is given by the universal norm:
			\begin{equation}\label{eq:univnorm}
				\|p\|_{\rm univ} \ =\ 
				\sup \left\{ \|\pi(p)\| \ : \ \pi = \pi_1 \ast \cdots \ast \pi_m \right\},
			\end{equation}
			where each $\pi_i : \POVM(n_i) \to  B(\HE)$ is a unital $*$-representation on a common Hilbert space $\HE$.
			
			Let $\pi = \pi_1 \ast \cdots \ast \pi_m$ be such a representation on $\HE$. For each $i$, let $Y_{i,j} = \pi_i(y_{i,j})$. By Lemma~\ref{l:repCpovm}, $Y_i = (Y_{i,1}, \dots, Y_{i,n_i-1}) \in \cD_{n_i}(\HE)$. Thus, the combined tuple $Y = (Y_1, \dots, Y_m)$ belongs to $\cD_L(\HE)$. Since $\pi(p) = p(Y)$ and the norm on $\POVM(\underline{n})$ is defined as the supremum over all such evaluations (see \eqref{eq:PsiHG} and Lemma~\ref{l:infvfin+}), it follows that
			\[
			\|\pi(p)\| \ =\ \|p(Y)\| \ \le\ \|p\|_{\POVM(\underline{n})}.
			\]
			Taking the supremum over all free product representations $\pi$ yields
			\begin{equation}
				\label{eq:astornot_1}
				\|p\|_{\rm univ} \ \le\ \|p\|_{\POVM(\underline{n})}.
			\end{equation}
			
			Conversely, the norm $\|p\|_{\POVM(\underline{n})}$ is achieved by taking the supremum of $\|p(Y)\|$ over all $Y \in \cD_L(\HE)$ and all Hilbert spaces $\HE$. Let $Y = (Y_1, \dots, Y_m) \in \cD_L(\HE)$. Then each $Y_i \in \cD_{n_i}(\HE)$, which, by the universal property of $\POVM(n_i)$, induces a unital $*$-representation $\rep_i: \POVM(n_i) \to  B(\HE)$ given by $\rep_i(y_{i,j}) = Y_{i,j}$. These representations naturally combine into a free product representation $\rep = \rep_1 \ast \cdots \ast \rep_m$ of the algebraic free product on $\HE$. For this representation, $\rep(p) = p(Y)$. Consequently,
			\[
			\|p(Y)\| \ =\ \|\rep(p)\| \ \le\ \|p\|_{\rm univ}.
			\]
			Taking the supremum over all $Y \in \cD_L(\HE)$ gives
			\begin{equation}
				\label{eq:astornot_2}
				\|p\|_{\POVM(\underline{n})} \ \le\ \|p\|_{\rm univ}.
			\end{equation}
			
			Combining \eqref{eq:astornot_1} and \eqref{eq:astornot_2}, we obtain $\|p\|_{\POVM(\underline{n})} = \|p\|_{\rm univ}$ for all $p \in \ast_{i=1}^m \povm(n_i)$. Since $\povm(\underline{n})$ is dense in $\POVM(\underline{n})$ by definition, and the algebraic free product $\ast_{i=1}^m \povm(n_i)$ is dense in $\POVM(n_1) \cbast \cdots \cbast \POVM(n_m)$, their respective $C^*$-completions coincide. Thus, $\POVM(\underline{n}) = \POVM(n_1) \cbast \cdots \cbast \POVM(n_m)$.
		\end{proof}

		\smallskip
		
		Theorem~\ref{t:posforpovmm} below is the main result of this section.
		
		\begin{thm}\label{t:posforpovmm}
			The ordered $*$-algebra $\povm(\underline{n})$ has a perfect Positivstellensatz. 
			Let $\pp\in  B(\cH)\otimes\povm(\underline{n})$ be of degree $2\vd+1$ (as an element of $ B(\cH) \otimes \C\la y\ra$). Then $\pp$ is positive if and only if the following hold.
			
			\begin{enumerate}[\rm (1)]\itemsep=5pt
				
				\item If $\cH$ is (separable and) infinite-dimensional, then there exist
				\[
				f,\,f_i,\,f_{i,j} \in  B(\cH)\otimes \povm(\underline{n})
				\]
				of degree $\vd$ such that
				\[
				\pp
				\ =\ 
				f^*f
				+
				\sum_{i=1}^m
				\left[
				\sum_{j=1}^{n_i-1} f_{i,j}^*\, y_{i,j}\, f_{i,j}
				+
				f_i^*\!\Big(1-\sum_{j=1}^{n_i-1} y_{i,j}\Big) f_i
				\right].
				\]
				
				\item If $\cH=\C^\nu$ is finite-dimensional, then there exist a positive
				integer $N$ and elements
				\[
				f_k,\,f_{i,k},\,f_{i,j,k}\in M_\nu(\C)\otimes \povm(\underline{n})
				\]
				of degree $\vd$ such that
				\[
				\pp
				\ =\ 
				\sum_{k=1}^N f_k^*f_k
				+
				\sum_{k=1}^N \sum_{i=1}^m
				\left[
				\sum_{j=1}^{n_i-1} f_{i,j,k}^*\, y_{i,j}\, f_{i,j,k}
				+
				f_{i,k}^*\!\Big(1-\sum_{j=1}^{n_i-1} y_{i,j}\Big) f_{i,k}
				\right].
				\]
			\end{enumerate}
		\end{thm}

		\begin{proof}
			The backward implication is immediate. We prove the forward implication. The polynomial  $\pp$ is identified with $\pp\in  B(\cH) \otimes \C \la y \ra,$ which in turn is identified with its
			image in $\BO(\HG).$ Thus,  the assumption that $\pp$ is positive means $\pp(E)\succeq 0$ for all $E\in \cD_L.$  Since
			also $\pp=\pp^*,$   Proposition~\ref{prop:aff} applies to $\pp$ yielding the desired conclusion. 
		\end{proof}

		\section{Proof of Theorem~\ref{t:main2}}\label{s:WS}
		
		This section is devoted to the proof of Theorem \ref{t:main2}. In a first step we reduce the problem from a free product of finite abelian groups to a free product of finite cyclic groups.
		
		\subsection{From abelian to cyclic} \label{ssec:abel2cyc}
		Our goal is to express positive elements of the group algebra
		\begin{equation*}
			\C[\mathbb{G}_{1}\ast \mathbb{G}_{2} \ast \cdots \ast \mathbb{G}_{m}]
			\ =\ 
			\C[\mathbb{G}_{1}] \ast \C[\mathbb{G}_{2}] \ast \cdots \ast \C[\mathbb{G}_{m}],
		\end{equation*}
		as sums of squares, where \(\mathbb{G}_{1},\mathbb{G}_{2},\dots,\mathbb{G}_{m}\) are finite abelian groups. The result is naturally interpreted in terms of polynomials. Doing so
		makes two novel aspects transparent. There are provable degree bounds and the  result holds even for polynomials with
		operator coefficients. 
		
		\smallskip
		
		For any finite abelian group $\mathbb{G},$ the $C^{*}$-algebra $\C[\mathbb{G}]$ is isomorphic to $\mathbb{C}^{|\mathbb{G}|}$, where $|\mathbb{G}|$ is the cardinality of $\mathbb{G}.$ Consequently,
		\[
		\C[\mathbb{G}_{1}] \ast \C[\mathbb{G}_{2}] \ast \cdots \ast \C[\mathbb{G}_{m}]
		\ \cong\ 
		\C[\Z_{n_{1}}] \ast \C[\Z_{n_{2}}] \ast \cdots \ast \C[\Z_{n_{m}}] \ =\ \C[\Z_{n_{1}} \ast \Z_{n_{2}} \ast \cdots \ast \Z_{n_{m}}],
		\]
		where $n_{i}$ is the cardinality of $\mathbb{G}_{i}.$ Moreover, such an isomorphism preserves both extent and positivity: an element
		\[
		p \ \in\  \C[\mathbb{G}_{1}\ast \mathbb{G}_{2} \ast \cdots \ast \mathbb{G}_{m}]
		\]
		has extent $\vd$ if and only if its image has extent $\vd$, and $p$ is positive if and only if its image is positive. Thus, it is enough to prove Theorem \ref{t:main2} for %
		a free product of finite cyclic groups.
		
		\subsection{Free product of finite cyclic groups} For the rest of this section, set 
		\[
		\groupW \ =\  \Z_{n_{1}} \ast \Z_{n_{2}} \ast \cdots \ast \Z_{n_{m}}.
		\]
		We shall express positive elements of the group algebra
		\begin{equation*}
			\C[\Z_{n_{1}} \ast \Z_{n_{2}} \ast \cdots \ast \Z_{n_{m}}] \ =\ \C[\Z_{n_1}] \ast \C[\Z_{n_2}] \ast \cdots \ast \C[\Z_{n_m}]
		\end{equation*}
		as sums of squares by applying the Positivstellensatz for \(\povm(\underline{n})\), namely Theorem~\ref{t:posforpovmm}.
		
		\smallskip
		
		Let $\cx_i$ denote a generator of $\Z_{n_i}.$ Thus, with multiplication as the
		group operation,  $\Z_{n_i}$ is, as a set, $\{\cx_{i}^j: 0\le j <n_i\}.$  Elements of $\groupW$ are words in $\cx=(\cx_1,\dots,\cx_m)$. A word $w \in \groupW$ has the form
		\[
		w\ =\  \cx_{i_1}^{j_1}  \, \cx_{i_2}^{j_2} \cdots \cx_{i_k}^{j_k}, \qquad 1\le j_\ell < n_{i_\ell}, \quad i_1\ne i_2\ne \cdots \ne i_k.
		\]
		\begin{remark} \label{def:total degree}\rm
			Here one defines the \df{total degree} of the word $w$
			to be $\sum_{\ell=1}^k j_\ell.$
			The \df{total degree} of a polynomial \(p\) in $\cx$ is then the largest total degree among all words \(w\) appearing with nonzero coefficient in \eqref{eq:poly}.
		\end{remark}
		
		Let \df{\(\cU(\underline{n})\)} denote the set of all \(m\)-tuples of unitary operators
		\[
		U=(U_1,\dots,U_m)
		\]
		on separable Hilbert space satisfying
		\[
		U_i^{\,n_i} \ =\ I,\qquad i=1,\dots,m.
		\]
		For such a tuple \(U\), define
		\[
		p(U)\ :=\ \sum_{w}^{\mathrm{finite}} P_w \otimes U^w.
		\]
		Then
		\[
		p^*(U)\ =\ p(U)^*.
		\]
		
		\smallskip
		
		Because \(\groupW\) is the free product of finite cyclic groups, every unitary representation \(\rep\in\Pi(\groupW)\) is uniquely determined by a tuple \(U=(U_1,\dots,U_m)\in \cU(\underline{n})\), and conversely every such tuple determines a unitary representation via
		\[
		\cx_i \mapsto U_i,\qquad i=1,\dots,m.
		\]
		Therefore, a polynomial $p \in   B(\cE) \otimes \C[\groupW]$ is positive  if and only if $p(U)\succeq0$ for all $U \in \cU(\underline{n}),$ and $p$ is hermitian  if 
		and only if $p(U)$ is hermitian for all $U \in \cU(\underline{n})$.
		
		\smallskip
		
		The norm on $\C[\groupW]$ defined by 
		\[
		\|p\| \ =\  \sup\left\{\, \big \|  p(U) \big \| :\  U \in \cU(\underline{n})\right\}
		\]
		(by considering the left regular action of $\C[\groupW]$ on $\ell^{2}(\groupW)$, it is easy to see that $\| \cdot \|$ is a norm, not just a semi-norm)
		satisfies the $C^{*}$ identity, $\|p(U)\|^2 =\| p(U)^*p(U)\|.$ Thus $\C[\groupW]$ (with this norm) is a pre-$C^{*}$-algebra whose completion is the free product (amalgamated over $\C$)  $C^{*}$-algebra,
		\index{$C^*(\groupW)$}
		\begin{equation} \label{eq:CWcompletion}
			C^*(\groupW)\ :=\  \C[\Z_{n_1}] \cbast \C[\Z_{n_2}] \cbast \cdots \cbast \C[\Z_{n_m}].
		\end{equation}

		\smallskip
		
		Let $\cE$ be any Hilbert space. The order (in the sense of positive semidefinite) and norm extend to polynomials in $ B(\cH)\otimes \C[\groupW]$ either by viewing $ B(\cH)\otimes \C[\groupW]$ as a subalgebra of the $C^{*}$-algebra $ B(\cH)\otimes C^*(\groupW)$
		{with the spatial (min) $C^*$-tensor product norm} 
		or more directly by the condition $p$ is positive if and only if  $p(U) \succeq 0$ for each $U \in \cU(\underline{n}).$

		\smallskip

		\subsection{Applying Boca's theorem}
		In this subsection Boca's theorem (\cite{Boc91, DK17}) is applied in anticipation of   transferring the Positivstellensatz of Theorem~\ref{t:posforpovmm} for $\povm(\underline{n})$ to a corresponding result for $\C[\groupW].$

		\subsubsection{The algebra \texorpdfstring{$\CZ$}{C[Z\_n]} in projection form}
		
		The group $*$-algebra $\CZ$ of the cyclic group $\Zn$ is canonically isomorphic to
		$\C^n$ as a $C^{*}$-algebra
		via the Fourier transform.  For our purposes it is convenient to present $\CZ$
		as the universal $*$-algebra generated by selfadjoint idempotents 
		\[
		\cq_1,\dots,\cq_n
		\]
		subject to the relations
		\begin{equation}\label{eq:CZn-P-presentation}
			\cq_i^* \ =\  \cq_i \ =\  \cq_i^2,\quad
			\cq_i \cq_j \ =\  0 \ (i\neq j),\quad
			\cq_1+\cdots+\cq_n \ =\  1.
		\end{equation}
		The $\cq_i$ are then minimal central projections and form a basis of $\CZ$.  To do so
		concretely  set
		\begin{equation}\label{eq:spectral-Pk}
			\cq_k \ =\  \frac1n\sum_{j=0}^{n-1} \omega^{-jk}\cx^j \in \CZ,
		\end{equation}
		where $\cx$ is a generator of the group $\Z_n$ and $\omega$ is a primitive $n$-th root of unity. Since the involution on $\CZ$ is given by, $(\cx^j)^*=\cx^{-j},$ it is readily
		checked that $\cq_k$ so defined satisfies the relations in equation~\eqref{eq:CZn-P-presentation}. 
		Moreover, if $U$ is a unitary operator satisfying $U^n=1,$ then 
		\[
		\cq_k(U) \ =\  \frac1n\sum_{j=0}^{n-1} \omega^{-jk} U^j,
		\]
		is the projection onto the spectral subspace of $U$ associated to its eigenvalue $\omega^k.$
		
		\smallskip 
		
		Let $\Omega_n: \CZ\to\povm(n)$ denote the unital linear map determined by $\Omega_n(\cq_j) = y_j\in \C \la y\ra,$
		for $j=1,2,\dots,n-1,$
		where $\C \la y\ra$ is identified with $\povm(n).$ 
		
		\begin{lemma}\label{cl:1}
			The linear map $\Omega_n:\CZ\to\povm(n)$ is completely positive.
		\end{lemma}
		
		\begin{proof}
			First note that $\Omega_n(\cq_j)=y_j\in \povm(n)$ is psd.
			Any $a\in M_k(\CZ)$ can be written uniquely as
			\[
			a \ =\  \sum_{j=1}^n a_j \otimes \cq_j,\qquad a_j \in M_k(\C)
			\]
			and $a$ is psd  if and only if each $a_i$ is psd.  In that case
			\[
			1_k\otimes \Omega_n(a) \ =\  \sum_{j=1}^n a_j\otimes \Omega_n(\cq_j) \succeq 0
			\]
			and the proof is complete.
		\end{proof}

		\subsubsection{Boca's theorem}
		
		As an initial observation, the mapping $\trace_n:\CZ\to \CZ$ on 
		diagonal $n\times n$ matrices 
		defined by
		\[
		\textstyle{\trace_n}(\sum_{j=1}^n a_j \cq_j) \ =\ \left [\frac{1}{n}\sum_{j=1}^n a_j \right ]  \, 1
		\]
		is a completely positive projection onto $\C\, 1.$  Writing 
		\begin{equation}
			\label{e:inverse-FT}
			\cx^j = \sum_{k=1}^n \omega^{j\, k} \cq_k
		\end{equation}
		it is evident that $\trace_n(\cx^j)=0$ for $1\le j<n$ and
		hence   the kernel of $\trace_n$ is the span of $\{\cx^j: 1\le j <n\}.$ 
		For $1\le i\le m,$ let $\Omega_i$ denote the linear map $\Omega_{n_i} : \C[\Z_{n_i}] \to \povm(n_i).$ 
		In the present setting Boca's theorem \cite{Boc91, DK17}  gives the  following result.
		
		\begin{proposition}
			\label{p:boca-appied}
			There is a ucp map $\Omega:\C[\groupW]\to \povm(\underline{n})$ 
			such that $\Omega |_{\C[\Z_{n_i}]} = \Omega_i$ and 
			\begin{equation}
				\label{eq:boca-applied}
				\Omega(\zz_1\cdots\zz_k) \ =\  \Omega_{i_1}(\zz_1) \cdots \Omega_{i_k}(\zz_k),
			\end{equation}
			when $\zz_\ell \in \ker \trace_{n_{i_\ell}}$ and $i_1\ne i_2\ne \cdots \ne i_k  .$
		\end{proposition}

		\begin{remark}\rm
			\label{r:boca-applied}
			\pushQED{\qed}
			From equation~\eqref{e:inverse-FT},
			\[
			\Omega(\cx_i^j) = \sum_{k=1}^{n_i} \Omega(\omega_i^{j\, k} \cp_{i,j})
			= \sum_{k=1}^{n_i-1} \omega_i^{j\, k} y_{i,k}  \, + \, (1-\sum_{k=1}^{n_i-1} y_{i,k}).
			\]
			In particular, $\Omega(\cx_i^j)$ is a polynomial of degree (at most) one in $\Cly.$
			
			For a reduced word  $w$  in $\C[\groupW],$ 
			\[
			w \ =\ \cx_{i_1}^{r_1}\cdots \cx_{i_k}^{r_k}, 
			\]
			(thus $1\le r_j < n_{i_j}$), we have, by Proposition~\ref{p:boca-appied},
			\[
			\Omega(w)   = \Omega_{i_1}(\cx_{i_1}^{r_1}) \cdots 
			\Omega_{i_k}(\cx_{i_k}^{r_k}).
			\]
			Further, for a polynomial  $p = \sum p_w w,$ where each $w$ is  reduced,
			\[
			\Omega(p) =\sum_w p_w \Omega(w).   \qedhere \popQED 
			\]
		\end{remark}

		\subsection{A splitting}
		For each $1\le i\le m$ let
		$\cp_{i,k}$ denote 
		\[
		\cp_{i,k} \ =\  \frac1n_i\sum_{j=0}^{n_i-1} \omega^{-jk}\cx_i^j \in \C[\Z_{n_i}]
		\]
		for $k=1,\dots,n_i.$ Compare with equation~\eqref{eq:spectral-Pk}. 
		Since $\C\la\cp\ra$
		is a free unital $*$-algebra generated by $\{\cp_{i,j}: 1\le i\le m, \, 1\le j\le n_i-1\}$ and $\povm(\underline{n})$ is also a unital $*$-algebra generated by $\{y_{i,j}: 1\le i\le m, \, 1\le j\le n_i-1\},$
		there is a unique unital $*$-homomorphism $\widetilde{s}:\C\la\cp\ra \to \povm(\underline{n})$ 
		determined by  $\widetilde{s}(\cp_{i,j})=y_{i,j}.$  Because the map $\Psi:\C \la y\ra\to  B(\HG)$ of equation~\eqref {eq:PsiHG} is faithful, $\widetilde{s}$ induces a $*$-unital map $s:\povm(\underline{n})\to \C[\groupW]$
		determined by $s(y_{i,j}) =\cp_{i,j},$ where $y_{i,j}$ is identified with $\Psi(y_{i,j})$.   
		
		\begin{lemma}\label{cl:2}
			The  unital $*$-homomorphism $s:\povm(\underline{n}) \to \C[\groupW]$ is surjective and splits $\Omega$ in the sense that $s\,(\Omega\,(z))=z$ for $z\in \C[\groupW].$ 
		\end{lemma}
		
		\begin{proof}
			Since $\widetilde{s}$ is surjective, so is $s.$  Since $\cp_{i,j}\in \C[\groupW]$ is mapped to $y_{i,j}\in \povm(\underline{n})$ (identified with $\C\la y\ra$) under $\Omega,$  it follows that
			$s(\Omega(\cp_{i,j})) =\cp_{i,j}.$   Thus,  using Proposition~\ref{p:boca-appied} and 
			Remark~\ref{r:boca-applied}, 
			\[
			s(\Omega(z_\ell)) \ =\  s(\Omega_\ell(z_\ell)) = z_\ell,
			\]
			for {$z_\ell \in  \ker \trace_{n_\ell} \subseteq\C[\Z_{n_\ell}]$}  and consequently, $s\,(\Omega\,(z))=z$ for a (reduced) word $z\in\C[\groupW]$
			by Remark~\ref{r:boca-applied}.
			From here the result follows by linearity.
		\end{proof}
		
		We are now ready to prove Theorem \ref{t:main2}.

		\begin{proof}[Proof of Theorem~\ref{t:main2}]
			Let $\vd$   denote the
			extent of $p$, set
			\[
			d \ =\ \lfloor \frac{\vd}{2} \rfloor,
			\]
			and note $\vd\le 2 d+1$.
			
			\smallskip
			
			We first claim that $\deg \Omega(p)\le \vd$ as a polynomial in the
			variables $y_{i,j}$. {Indeed, by Remark~\ref{r:boca-applied}, if
				\[
				w \ =\ \cx_{i_1}^{r_1}\cdots \cx_{i_k}^{r_k}
				\]
				is a reduced word of extent $k$, then $\Omega(\cx_{i_j})$ has degree at most one
				in $\Cly$ by Remark~\ref{r:boca-applied} and 
				therefore $\Omega(w)$ has degree at most $k$.} Since every word
			appearing in $p$ has extent at most $\vd$, it follows that
			\[
			\deg \Omega(p) \ \le\  \vd \ \le\  2 d+1.
			\]
			
			\smallskip
			
			Applying Theorem~\ref{t:posforpovmm} to $\pp=\Omega(p)$ therefore yields a
			weighted sum-of-squares representation in which all coefficient
			polynomials $f_*,g_*$ have degree at most $d$. Now $q(y_{i,j})=\cp_{i,j}$,
			and each $\cp_{i,j}$ belongs to the single factor
			$\C[\Z_{n_i}]\subseteq \C[\groupW]$. Consequently, if $r$ is a monomial
			of degree $t$ in the variables $y_{i,j}$, then $s(r)$ is a product of
			$t$ elements taken from the factors $\C[\Z_{n_i}]$, and after reducing
			adjacent letters from the same factor one obtains a linear combination of
			reduced words of extent at most $t$. Hence
			\[
			\text{extent of } s(r) \ \le\  \deg r
			\]
			for every polynomial $r$, and in particular $s(g_{k}), s(f_{i,j,k}),$ and $s(f_{i,k})$ all have extent at most $d.$
			Finally, writing
			\[
			\cp_{i,n_i} \ =\ 1-\sum_{j=1}^{n_i-1}\cp_{i,j},
			\]
			and defining
			\[
			\ch_{i,j,k}=\cp_{i,j}\,s(f_{i,j,k}) \quad (1\le j\le n_i-1),\qquad
			\ch_{i,n_i,k}=\cp_{i,n_i}\,s(f_{i,k}),
			\]
			we obtain
			\[
			p \ =\ \sum_k s(g_k)^*s(g_k)+
			\sum_k\sum_{i=1}^m\sum_{j=1}^{n_i}\ch_{i,j,k}^*\ch_{i,j,k}.
			\]
			Since each $\cp_{i,j}$ has extent at most $1$, it follows that
			\[
			\text{extent of } \ch_{i,j,k} \ \le\  d+1 \ =\ \lfloor \frac{\vd}{2} \rfloor +1.
			\]
			Thus all summands have extent at most $\lfloor \frac{\vd}{2} \rfloor +1$. 
		\end{proof}

	\end{document}